\documentclass[a4paper,12pt]{amsart}
\usepackage{amsfonts,amscd,amsmath,amssymb,tikz}
\usepackage{pifont}
\usepackage{youngtab}
\usepackage{ytableau}

\newcommand{\NN}{{\mathbb{N}}}

\newcommand{\fS}{{\mathfrak{S}}}

\let\la=\lambda

\newcommand{\6}{^}

\newtheorem{thm}{Theorem}[section]
\newtheorem{lem}[thm]{Lemma}
\newtheorem{cor}[thm]{Corollary}
\newtheorem{prop}[thm]{Proposition}
\newtheorem{define}[thm]{Definition}

\newtheorem*{conjA}{Conjecture A}
\newtheorem*{conjB}{Conjecture B}

\theoremstyle{definition}
\newtheorem{nota}[thm]{Notation}

\theoremstyle{remark}
\newtheorem{rem}[thm]{Remark}
\newtheorem{example}[thm]{Example}

\begin{document}

\title{McKay bijections and character degrees}

\author{Eugenio Giannelli}
\address{Dipartimento di Matematica e Informatica Ulisse Dini, Universita degli Studi di Firenze}
\email{eugenio.giannelli@unifi.it}



\begin{abstract}
We propose a new refinement of the McKay conjecture 
and we prove it for symmetric groups. 
\end{abstract}

\maketitle


\section{Introduction}

Let $G$ be a finite group, let \( p \) be a fixed prime number, and let \( P \) be a Sylow \( p \)-subgroup of \( G \). As usual, we denote by \( \mathrm{Irr}_{p'}(G) \) the set of irreducible characters of \( G \) whose degrees are coprime to \( p \). The McKay conjecture, first stated in \cite{McKay}, asserts that
\[
|\mathrm{Irr}_{p'}(G)| = |\mathrm{Irr}_{p'}(N_G(P))|.
\]
This conjecture has been one of the central problems in the representation theory of finite groups for nearly 60 years. In a groundbreaking article \cite{IMN}, the McKay conjecture was reduced to a collection of deeper statements concerning finite simple groups.
This reduction marked a major breakthrough: the conjecture was subsequently proven for the primes \( 2 \) and \( 3 \) in \cite{MS} and \cite{S23}, respectively, and was ultimately confirmed in full generality in \cite{CS} in 2024.

Several refinements of the McKay conjecture have been proposed over the years, aiming to provide a deeper understanding of the relationship between \( G \) and \( N_G(P) \). Notably, Alperin formulated a stronger version 
\cite{Alperin}, which considers the distribution of irreducible characters across \( p \)-blocks. Isaacs and Navarro proposed a second refinement keeping track of the $\mathrm{mod}$ $p$-congruences of the degrees of corresponding characters \cite{IN}.
Further developments by Navarro introduced additional refinements of both the McKay and the Alperin-McKay conjectures \cite{N04}, which take into account the action of Galois automorphisms on irreducible characters. These newer conjectures not only generalize McKay's original statement but also suggest a rich and intricate interplay between local and global representation-theoretic invariants.

In this article, we propose a further strengthening of the McKay conjecture.

\begin{conjA}
Let $G$ be a finite group, let $p$ be a prime number and let $P$ be a Sylow $p$-subgroup of $G$. Then there exists a bijection $$\varepsilon: \mathrm{Irr}_{p'}(G)\rightarrow \mathrm{Irr}_{p'}(N_G(P)),$$
such that $\varepsilon(\chi)(1)\leq\chi(1)$, for all $\chi\in\ \mathrm{Irr}_{p'}(G).$
\end{conjA}

A first straightforward observation is that Conjecture A holds whenever $G$ admits a selfnormalising Sylow $p$-subgroup. 
Moreover, Conjecture A has been recently reduced to finite simple groups in \cite{HMN} and, in that same article, has been verified for the prime $p=2$. In particular, the authors show that our Conjecture A implies an interesting statement concerning the sum of the squares of the degrees of $p'$-degree characters of a finite group \cite[Conjecture A]{HMN}. 
We also wish to point out that Conjecture A is known to hold for $p$-solvable groups. More precisely the combination of results of Geck \cite{Geck} and Rizo \cite[Theorem A]{Rizo}, extending previous work of Turull on solvable groups \cite[Main Theorem]{Tu}, shows that when $G$ is a $p$-solvable group there exists a bijection $f:\mathrm{Irr}_{p'}(G)\rightarrow\mathrm{Irr}_{p'}(N_G(P))$ such that $f(\chi)(1)$ divides $\chi(1)$, for every $\chi\in\mathrm{Irr}_{p'}(G)$. 
This stronger statement has no chances to hold in general. For instance it fails in the case of the symmetric group $S_7$ and the prime $3$. 

In this article, we bring new evidence in support of Conjecture A by focusing our attention on symmetric groups.

\begin{thm}\label{thm: MAIN}
Conjecture A holds for symmetric groups and any prime number. 
\end{thm} 



As we observed above, Theorem \ref{thm: MAIN} is known to hold for $p=2$. 
For this reason, in this paper, we focus our attention on odd primes.

To prove Theorem  \ref{thm: MAIN} in this setting, we employ a combination of techniques. First, we extend and refine the description of the irreducible characters of the normaliser \( N_{S_n}(P) \), initially developed in \cite{GANT} (see also \cite{Tendi}). A second key ingredient, which we wish to emphasize, is the recent work appeared in \cite{GL2, GL3}, which provides a detailed description of the irreducible constituents arising from the restriction of irreducible characters of symmetric groups to Sylow \( p \)-subgroups. These new developments are combined with classical results on minimal degrees of irreducible characters of symmetric groups, established by Rasala in \cite{Rasala}, as well as Olsson's theory of \( p \)-core towers, as presented in \cite{OlssonBook}.

\smallskip

%
%
%

We conclude by noting that the results of this paper can be readily adapted to show the existence of a bijection between the height zero irreducible characters of a $p$-block of maximal defect of $S_n$ and those of its Brauer correspondent. Extending this result to arbitrary $p$-blocks of symmetric groups lies beyond the scope of this article, as it will require more sophisticated algebraic and combinatorial techniques, and will be the subject of future investigation.  
Nevertheless, we believe the following may hold.

\noindent (For a $p$-block $C$ of a finite group $H$, we let $\mathrm{Irr}_{0}(C)$ denote the set of height zero irreducible characters of $H$ lying in $C$.)

\begin{conjB}
Let $G$ be a finite group, let $p$ be a prime number, let $B$ be a $p$-block of $G$, and let $b$ be its Brauer correspondent.  
Then there exists a bijection  
$$\varepsilon: \mathrm{Irr}_{0}(B) \rightarrow \mathrm{Irr}_{0}(b),$$  
such that $\varepsilon(\chi)(1) \leq \chi(1)$ for all $\chi \in  \mathrm{Irr}_{0}(B)$.
\end{conjB}

Encouraging evidence in support of Conjecture B is presented in \cite{MR}, where the above statement is proved for $p$-solvable groups. Moreover, Conjecture B has been recently verified in \cite{Lin} for blocks with cyclic or Klein four defect group.

\subsection*{Acknowledgments}
The author's research is funded by: the European Union Next Generation EU, M4C1, CUP B53D23009410006, PRIN 2022 - 2022PSTWLB Group Theory and Applications; and the INDAM-GNSAGA Project CUP E53C24001950001. 

 We wish to thank S. Law, M. Martinez, G. Navarro, F. Pediconi, C. Vallejo and M. Wildon for several useful conversations, and G. Malle for comments and corrections on a previous version of this paper. 
 We are also very grateful to the referee for several comments and suggestions that improved the readability of the article.

\section{Background and Notation}\label{sec: 2}
In this article, for any $x,y\in\mathbb{N}$ we use the symbol $[x,y]$ to denote the subset $\{x,x+1,\ldots, y\}$ of $\mathbb{N}$. We regard $[x,y]$ as the empty set whenever $x>y$.

We start by briefly recalling some standard facts on the character theory of symmetric groups. We refer the reader to the usual references \cite{JK} and \cite{OlssonBook} for complete details. 
We let $\mathcal{P}(n)$ denote the set of partitions of $n$.
The irreducible characters of the symmetric group $S_n$ are naturally labelled by partitions of $n$. For a partition $\lambda=(\lambda_1,\ldots, \lambda_s)\in\mathcal{P}(n)$, we set $|\lambda|=n$, we denote by $\lambda'$ its conjugate partition and by $\chi^\lambda$ the irreducible character of $S_n$ corresponding to $\lambda$. We set $\ell(\lambda)=(\lambda')_1$ (i.e. $\ell(\lambda)$ equals the number of parts of $\lambda$). 
As a convention, we regard $\lambda_s=0$ for every $s>\ell(\lambda)$.
Given $n,r\in\mathbb{N}$ and 
partitions $\gamma\in\mathcal{P}(r)$ and $\lambda\in\mathcal{P}(n)$ we write $\gamma\subseteq\lambda$ if $\gamma_\ell\leq\lambda_\ell$ for every $\ell\in\mathbb{N}$ (equivalently, $\gamma\subseteq\lambda$ if the Young diagram of $\gamma$ is contained in that of $\lambda$). We write $\lambda+\gamma$ to denote the partition $\mu\in\mathcal{P}(n+r)$ such that $\mu_\ell=\lambda_\ell+\gamma_\ell$, for every $\ell\in\mathbb{N}$. Finally we use the symbol $\lambda\cup\gamma$ to denote the partition of $n+r$ obtained by reordering the parts of the composition $(\lambda_1,\ldots,\lambda_{\ell(\lambda)}, \gamma_1,\ldots, \gamma_{\ell(\gamma)}).$
Moreover, for any $t\in [1,n]$ we denote by $\mathcal{B}_n(t)$ the subset of $\mathcal{P}(n)$ defined by
$$\mathcal{B}_n(t)=\{\lambda\in\mathcal{P}(n)\ |\ \lambda_1, \ell(\lambda)\leq t\}.$$
In other words, $\mathcal{B}_n(t)$ consists of those partitions of $n$ whose Young diagram fits into a $(t\times t)$-square grid. 
For a subset $A\subseteq\mathcal{P}(n)$ we let $A^{\circ}=\{\lambda, \lambda'\ |\ \lambda\in A\}$. Notice that $\mathcal{B}_n(t)^{\circ}=\mathcal{B}_n(t)$ for any $n,t\in\mathbb{N}$.

We recall that the Young diagram $Y(\lambda)$ is defined as follows: 
$$Y(\lambda)=\{(i,j)\in\NN\times\NN\mid i\in [1,\ell(\lambda)],\ j\in [1,\la_i]\}.$$

\medskip

\noindent Every element $(r,c)\in Y(\lambda)$ is called a \textit{node} of $Y(\lambda)$. We denote by $h_{r,c}(\lambda)$ the \textit{hook} associated to the node $(r,c)$, and we let $|h_{r,c}(\lambda)|$ be the size (or length) of the hook $h_{r,c}(\lambda)$. 
We recall that $h_{r,c}(\lambda)$ is the subset of $Y(\lambda)$ defined by
$$h_{r,c}(\la)
  =\{(r,y)\in Y(\lambda)\mid y\geq c\}\cup\{(x,c)\in Y(\lambda)\mid x\geq r\}.$$
  
Given $(i,j)\in Y(\lambda)$, we say that $h_{i,j}(\lambda)$ is an $e$-hook of $\lambda$ if its size $|h_{i,j}(\lambda)|$ is divisible by a given integer $e$. We let $\mathcal{H}^{e}(\lambda)$ be the set consisting of all the $e$-hooks of $\lambda$. 
A partition $\gamma$ is called an $e$-core if $\mathcal{H}^{e}(\lambda)=\emptyset$. The $e$-core of $\lambda$ is the $e$-core partition obtained from $\lambda$ by subsequently removing $e$-hooks. This is denoted by $C_e(\lambda)$.
The removal of a hook from the Young diagram of a partition corresponds to the removal of its rim-hook. 
 Finally, we let $Q_e(\lambda)=(\lambda^{(0)},\lambda^{(1)},\ldots,\lambda^{(e-1)})$ be the $e$-quotient of $\lambda$. The \textit{size} of the $e$-quotient is denoted by $|Q_e(\lambda)|$ and it is defined as the integer $|\lambda^{(0)}|+|\lambda^{(1)}|+\cdots +|\lambda^{(e-1)}|$.
 We refer the reader to \cite[Chapter I]{OlssonBook} for precise definitions and an extensive discussion on these combinatorial objects. We just recall that $|\mathcal{H}^e(\lambda)|=|Q_e(\lambda)|$ and that $|\lambda|=|C_e(\lambda)|+e|Q_e(\lambda)|$.

 We say that a partition $\lambda\in\mathcal{P}(n)$ is a \textit{hook partition} if $\lambda=(n-x,1^x)$ for some $x\in [0,n-1]$, and we denote by $\mathcal{H}(n)$ the subset of $\mathcal{P}(n)$ consisting of all hook partitions of $n$. 
 
Let $n_1,n_2,\ldots, n_t\in\mathbb{N}$ be such that $n_1+n_2+\cdots+n_t=n$, let $Y=S_{n_1}\times S_{n_2}\times\cdots\times S_{n_t}$ be the corresponding Young subgroup of $S_n$, and let $\lambda\in\mathcal{P}(n)$ and $\mu_i\in\mathcal{P}(n_i)$ for every $i\in [1,t]$. Then we let $$\mathrm{LR}(\lambda;\mu_1,\mu_2,\ldots, \mu_t),$$
be the multiplicity of $\chi^{\mu_1}\times\chi^{\mu_2}\times\cdots\times\chi^{\mu_t}\in\mathrm{Irr}(Y)$ as an irreducible constituent of $(\chi^\lambda)_Y$. These integers are known as Littlewood-Richardson coefficients.
We refer the reader to \cite[Chapter 17]{James} for a complete discussion about the Littlewood-Richardson rule. In the following lemma we just recall a few well-known facts that we will use later in this article. 

\begin{lem}\label{lem: LR}
Let $r,x\in\mathbb{N}$ be such that $r<x$ and let $n=x+r$. Let $\lambda\in\mathcal{P}(n)$ and let $\gamma\in\mathcal{P}(r)$. 
The following hold. 
\begin{itemize}
\item[(a)] There exists $\mu\in\mathcal{P}(x)$ such that $\mathrm{LR}(\lambda;\mu, \gamma)\neq 0$ if and only if $\gamma\subseteq\lambda$.
\item[(b)] Suppose that $x\geq 3$. Then there exists $\mu\in\mathcal{P}(x)\smallsetminus\{(x)\}^{\circ}$ such that $\mathrm{LR}(\lambda;\mu, \gamma)\neq 0$ if and only if $\gamma\subseteq\lambda$ and $\lambda\notin\{\gamma+(x), \gamma\cup(1^{x})\}$.
\end{itemize}
\end{lem}
\begin{proof}
Statement $(a)$ is an immediate consequence of the Littlewood-Richardson rule. Similarly statement $(b)$ easily follows from \cite[Lemma 4.4]{BK}. 
\end{proof}
 
For a prime number $p$, we will denote by $\mathcal{P}_{p'}(n)$ the subset of $\mathcal{P}(n)$ consisting of partitions $\lambda$ such that $\chi^\lambda\in\mathrm{Irr}_{p'}(S_n)$. 
The following proposition is a direct consequence of \cite[Proposition 6.4]{OlssonBook} and it describes the elements of $\mathcal{P}_{p'}(n)$.

\begin{prop}\label{prop: pdivisibility}
 Let $n$ and $k$ be positive integers such that $n=ap\6k+r$, for some $r<p^k$ and some $a\in [1,p-1]$.
 Let $\lambda$ be a partition of $n$ and let $\gamma=C_{p\6k}(\lambda)$.
Then 
 $\chi^\lambda\in\mathrm{Irr}_{p'}(S_n)$ if and only if $|\mathcal{H}^{p^k}(\lambda)|=a$ and $\chi^{\gamma}\in\mathrm{Irr}_{p'}(S_{r})$.
 \end{prop}

Proposition \ref{prop: pdivisibility} shows that if $n=ap^k+r$ with $r<p^k$, then $$\mathrm{Irr}_{p'}(S_n)=\bigcup_{\gamma\in\mathcal{P}_{p'}(r)}\mathrm{Irr}(S_n\ |\ \gamma),$$
where $\mathrm{Irr}(S_n\ |\ \gamma)=\{\chi^\lambda\in\mathrm{Irr}(S_n)\ |\ C_{p^k}(\lambda)=\gamma\}$. We are also able to control $|\mathrm{Irr}(S_n\ |\ \gamma)|$ for any $\gamma\in\mathcal{P}(r)$.
The following statement is a direct consequence of \cite[Proposition 3.7]{OlssonBook}. 

\begin{lem}\label{lem: McKay}
Let $n$ and $k$ be positive integers such that $n=ap\6k+r$, for some $r<p^k$ and some $a\in [1,p-1]$. For any $\gamma\in\mathcal{P}(r)$ we have that $$|\mathrm{Irr}(S_n\ |\ \gamma)|=|\mathrm{Irr}_{p'}(S_{ap^k})|.$$
In particular, given $P\in\mathrm{Syl}_{p}(S_{ap^k})$ we have that $|\mathrm{Irr}(S_n\ |\ \gamma)|=|\mathrm{Irr}_{p'}(N_{S_{ap^k}}(P))|$
\end{lem}
\begin{proof}
Every partition is uniquely determined by its $p^k$-core and its $p^k$-quotient, by \cite[Proposition 3.7]{OlssonBook}. The set $\mathrm{Irr}(S_n\ |\ \gamma)$ consists of all those characters labelled by partitions with $p^k$-core $\gamma$ and $p^k$-quotient of size $a$. Similarly, Proposition \ref{prop: pdivisibility} shows that $\mathrm{Irr}_{p'}(S_{ap^k})$ consists of all those characters labelled by partitions with empty $p^k$-core and $p^k$-quotient of size $a$. Hence, we conclude that $|\mathrm{Irr}(S_n\ |\ \gamma)|=|\mathrm{Irr}_{p'}(S_{ap^k})|$. The second statement is an immediate consequence of the validity of the McKay conjecture, first proved for symmetric groups in \cite{OlssonMcKay}.
\end{proof}

The following lemma is a consequence of a beautiful combinatorial result proved by Bessenrodt in \cite{Bess}. 

\begin{lem}\label{lem: bess}
Let $n=x+r$ and let $\gamma\in\mathcal{P}(r)$ be an $x$-core. Let $\mathcal{P}(n\ |\ \gamma)$ be the set of all partitions of $n$ with $x$-core equal to $\gamma$. Then 
$\mathcal{P}(n\ |\ \gamma)=\{\lambda_0,\lambda_1,\ldots,\lambda_{x-1}\},$
where for every $i\in [0,x-1]$ we have that $[\chi^{(x-i,1^i)}\times\chi^\gamma, (\chi^{\lambda_i})_{S_x\times S_r}]\neq 0$.
\end{lem}
\begin{proof}
Since $\gamma$ is an $x$-core, \cite[Theorem 1.1]{Bess} shows that $\mathcal{P}(n\ |\ \gamma)=\{\lambda_0,\lambda_1,\ldots,\lambda_{x-1}\}$ where for every $\ell\in [0,x-1]$ there exists a pair $(i,j)$ such that $h_{i,j}(\lambda_\ell)=(x-\ell,1^\ell)$. It follows that $\mathrm{LR}(\lambda_i; (x-\ell,1\6\ell), \gamma)\neq 0$ and therefore that $\chi^{(x-\ell,1^\ell)}\times\chi^\gamma$ is an irreducible constituent of the restriction of $\chi^{\lambda_\ell}$ to $S_x\times S_r$. 
\end{proof}

\subsection{Abacus combinatorics and $\beta$-sets}\label{sec: beta}
To prove our main results we need to control the $p^k$-hooks of several partitions. 
The best setting to do that is that of $\beta$-sets. In particular, it is convenient to visualize these on the so called \textit{James' abacus} for partitions. 
We devote this short section to recalling the main features of these combinatorial objects. 
We refer the reader to \cite{OlssonBook} for a complete account.

We call $\beta$-set any non-empty finite subset $X$ of $\mathbb{N}\cup \{0\}$. Given a $\beta$-set $X=\{h_1,\ldots, h_t\}$ with $h_1>h_2>\cdots>h_t$, we say that $$P(X)=(h_1-(t-1), h_2-(t-2),\ldots, h_t),$$ is the partition corresponding to $X$. 
Given $s\in\mathbb{N}$ we denote by $X^{+s}$ the set defined by $$X^{+s}=\{h_1+s,\ldots, h_t+s, s-1,s-2,\ldots, 1, 0\}.$$ 
Given $\lambda$ a partition of $n$, we say that $X$ is a $\beta$-set for $\lambda$ if $P(X)=\lambda$. 
We know from \cite[Proposition 1.3]{OlssonBook} that $X_\lambda=\{|h_{i,1}(\lambda)|\ |\ i\in [1,\ell(\lambda)]\}$ is a $\beta$-set for $\lambda$, and that a $\beta$-set $Y$ is a $\beta$-set for $\lambda$ if and only if $Y=(X_\lambda)^{+s}$, for some $s\in\mathbb{N}$. 

It is particularly convenient to use $\beta$-sets to deduce information about the hook lengths of their corresponding partitions. For instance from \cite[Corollary 1.5]{OlssonBook} we have that the following holds.

\begin{prop}\label{prop: hooksbetaset}
Let $\lambda$ be a partition of $n$, let $X$ be a $\beta$-set for $\lambda$ and let $h\in\mathbb{N}$. Then $h$ is the length of a hook in $\lambda$ if and only if there exist $x,y\in\mathbb{N}\cup\{0\}$ such that 
$$x\in X,\ \ y\notin X,\ \ \text{and}\ \ x-y=h.$$
\end{prop}

Proposition \ref{prop: hooksbetaset} shows that for every $x\in X$ and $y\in\mathbb{N}\smallsetminus X$ such that $x-y>0$, there is a corresponding hook in $\lambda$ of size $x-y$. In what follows we are going to use the notation $H(x,y)$ to denote such a hook of $\lambda$. Finally, we will often use the following reformulation of \cite[Proposition 1.8]{OlssonBook}. 

\begin{prop}\label{prop: removebetaset}
Let $\lambda$ be a partition of $n$, let $X$ be a $\beta$-set for $\lambda$ and let $H(x,y)$ be a hook in $\lambda$. Then $(X\smallsetminus\{x\})\cup\{y\}$ is a $\beta$-set for $\lambda\setminus H(x,y)$.
\end{prop}


We pause to give a small example to familiarize with the objects introduced so far. 

\begin{example}\label{ex: 1}
Let $\lambda=(3,2,2)$, then $X=\{7,5,4,1,0\}$ is a $\beta$-set for $\lambda$. We observe that $|h_{1,2}(\lambda)|=4$, and in fact, by direct inspection of the hook lengths in the Young diagram $Y(\lambda)$ we see that $h_{1,2}(\lambda)$ is the only hook of length $4$ in $\lambda$. This can be easily read off the $\beta$-set $X$ as well. In fact, $3\notin X$, $7\in X$ and $7-3=4$. This shows that $H(7,3)$ is a hook of length $4$ in $\lambda$. Moreover, if $x\in X$ and $y\notin X$ then $|H(x,y)|=4$ implies $x=7$ and $y=3$. We conclude that $H(7,3)$ is the only hook of length $4$ in $\lambda$ and therefore that $H(7,3)=h_{1,2}(\lambda)$. If we now let $Y=(X\smallsetminus\{7\})\cup\{3\}$, we have that $Y=\{5,4,3,1,0\}$ is a $\beta$-set for the partition $(1,1,1)$. As explained in the discussion after Proposition \ref{prop: hooksbetaset}, we have that $(1,1,1)=\lambda\setminus h_{1,2}(\lambda)$. 
\end{example}

We now fix $r\in\mathbb{N}$. Given $X$ a $\beta$-set and $\ell\in [0,r-1]$, we let $$X^\ell=\{a\in\mathbb{N}\cup\{0\}\ |\ r\cdot a+\ell\in X\}.$$
In particular we have that $X=\{r\cdot a+\ell\ |\ \ell\in [0,r-1]\ \text{and}\ a\in X^\ell\}$. 
As explained in \cite[Chapter 1]{OlssonBook} we have that if
 $\lambda=P(X)$, then $C_r(\lambda)=P(Y)$, where $$Y=\{r\cdot a+\ell\ |\ \ell\in [0,r-1]\ \text{and}\ 0\leq a\leq |X^\ell|-1\}.$$
Moreover, $(P(X^0),P(X^1),\ldots, P(X^{r-1}))=:Q_{r}(\lambda)$ defines the $r$-quotient of $\lambda$. 

\begin{rem}\label{rem: quotient}
Notice that the definition of $Q_{r}(\lambda)$ given above depends on the choice of the $\beta$-set $X$. In particular, we have that the $r$-quotient is defined up to a cyclic permutation of its components. 
In \cite[Chapter 1]{OlssonBook} this problem is resolved by always choosing $\beta$-sets of sizes divisible by $r$. We care to stress that in this paper we will not need this level of accuracy, because we will always be interested only in the sizes of the components of certain $r$-quotients, not in their relative positions. 
\end{rem}

All these considerations can be easily visualized by placing the elements of $X$ on a $r$-abacus, in the following way. The $r$-abacus has $r$ runners going from north to south numbered $0,1,\ldots, r-1$. On each runner the positions are labelled in increasing order $0,1,2,\ldots$ from top to bottom. The $r$-abacus configuration corresponding to $X$ is then obtained by placing a bead in position $a$ of runner $\ell$ if and only if $r\cdot a+\ell\in X$. In particular, the set $X^\ell$ records the positions of the beads lying on the $\ell$-th runner. 
 It is now particularly easy to spot $r$-hooks, in fact $\lambda:=P(X)$ admits a $r$-hook if and only if in the $r$-abacus configuration corresponding to $X$ there is a runner admitting a bead and a vacant position somewhere above that bead.
It is now clear that the $r$-abacus configuration corresponding to the $r$-core of $\lambda$ is obtained from that of $\lambda$ by sliding all beads as north as possible on each runner. This in turn corresponds to our first statement, namely that the $\beta$-set $Y$, as described above, is a $\beta$-set for $C_r(\lambda)$. 

\smallskip

\subsection{Characters of wreath products}\label{sec: 2.2}
We conclude this section by fixing our notation for characters of wreath products. This is based on \cite[Chapter 4]{JK}.
Given a finite group $G$  and a natural number $n$, we denote by $G^{\times n}$ the direct product of $n$ copies of $G$. For any subgroup $H\le\fS_n$, the permutation action of $\fS_n$ on the direct factors of $G^{\times n}$ induces an action of $\fS_n$ (and therefore of $H\le \fS_n$) via automorphisms of $G^{\times n}$, giving the wreath product $G\wr H= G^{\times n}\rtimes H$. The normal subgroup $G^{\times n}$ is sometimes called the base group of the wreath product $G\wr H$. 

Following \cite[Chapter 4]{JK}, the elements of $G\wr H$ are denoted by $(g_1,\dotsc,g_n;h)$ for $g_i\in G$ and $h\in H$. 
Let $V$ be a $\mathbb{C}G$--module and let $\phi$ be the character afforded by $V$. 
We let $V^{\otimes n}=V\otimes\cdots\otimes V$ ($n$ copies) be the corresponding $\mathbb{C}G^{\times n}$--module. The left action of $G\wr H$ on $V^{\otimes n}$ defined by linearly extending
$$(g_1,\dotsc,g_n;h)\ :\quad v_1\otimes \cdots\otimes v_n \longmapsto g_1v_{h^{-1}(1)}\otimes\cdots\otimes g_nv_{h^{-1}(n)}$$
turns $V^{\otimes n}$ into a $\mathbb{C}(G\wr H)$--module, which we denote by $\widetilde{V^{\otimes n}}$ (see \cite[(4.3.7)]{JK}). We let $\tilde{\phi}$ denote the character afforded by the representation $\widetilde{V^{\otimes n}}$. For any character $\psi$ of $H$, we abuse notation and let $\psi$ also denote its inflation to $G\wr H$. Finally, we introduce the symbol
$$\mathcal{X}(\phi;\psi)=\tilde{\phi}\cdot\psi$$
to denote the  character of $G\wr H$ obtained as the
product of $\tilde{\phi}$ and $\psi$.

Let $\mathrm{Lin}(G)$ be the set of linear characters of $G$. 
Basic Clifford theory \cite[Chapter 6]{IBook} shows that the set $\mathrm{Lin}(G\wr H)$ has the following form:
$$\mathrm{Lin}(G\wr H)=\{\mathcal{X}(\phi;\psi)\ |\ \phi\in\mathrm{Lin}(G)\ \text{and}\ \psi\in\mathrm{Lin}(H)\}.$$

\section{Preliminary results}

In this section, we collect several results primarily concerning the degrees of irreducible characters of symmetric groups and their Sylow normalizers. 
Some of these statements follow from known results, while others represent genuine extensions that may be of independent interest beyond their role in the proof of Theorem \ref{thm: MAIN}.

We begin by fixing the notation for Sylow subgroups and their normalizers. Much of this notation and background is borrowed from \cite{GANT}.

\subsection{Sylow subgroups and normalisers}\label{subsec: normalisers}
We use this section to recall some fundamental facts about irreducible Sylow $p$-subgroups of symmetric groups and their normalizers. 

Let $n\in\mathbb{N}$ and let $n=\sum_{j=1}^ta_jp^{k_j}+a_0$ be the $p$-adic expansion of $n$, with $k_t>k_{t-1}>\cdots>k_1\geq 1$. The integer $t$ is called $p$-\textit{adic length} of $n$. It is well known that if we denote by $P_n$ a fixed Sylow $p$-subgroup of $S_n$ then 
$$P_n=(P_{p^{k_t}})^{\times a_t}\times \cdots\times (P_{p^{k_1}})^{\times a_1},$$
where $P_{p^\ell}=P_{p^{\ell-1}}\wr P_p\cong C_p\wr C_p\wr\cdots\wr C_p$ is the iterated wreath product of $\ell$ copies of $C_p$, the cyclic group of order $p$.

Unless otherwise stated, for the rest of this article we will use the symbol $N_n$ to denote the subgroup $N_{S_n}(P_n)$ of $S_n$. It is well known that $N_{p^k}\cong P_{p^k}\rtimes (C_{p-1})^{\times k}$ and that
 $N_{ap^k}=N_{p^{k}}\wr S_{a}$ for any $a\in [1,p-1]$. Moreover, we have that 
$$N_n=N_{a_tp^{k_t}}\times \cdots\times N_{a_2p^{k_2}}\times N_{a_1p^{k_1}}\times S_{a_0}.$$
Thus $\mathrm{Irr}(N_n)=\{\theta_1\times\cdots\times\theta_t\times\chi^\gamma\ |\ \theta_i\in\mathrm{Irr}(N_{a_ip^{k_i}})\ \text{for any}\ i\in [1,t],\ \text{and}\ \gamma\in\mathcal{P}(a_0)\}$.

We now aim at describing the irreducible $p'$-degree characters of $N_{ap\6k}\cong N_{p\6k}\wr S_a$. In order to do this, we first need to recall a few properties of the set $\mathrm{Lin}(P_{p\6k})$. 
Following the notation introduced in \cite[Section 2.1]{GANT} we let $C_p=\langle g\rangle$ be a cyclic group of order $p$ generated by $g$. 
Let $\mathbb{C}_p$ denote the set of complex $p$th roots of unity and let 
$$\mathrm{Irr}(C_p)=\mathrm{Lin}(C_p)=\{\phi_z\ |\ z\in\mathbb{C}_p\},$$ where $\phi_z(g)=z$ for all $z\in \mathbb{C}_p$. 
This shows that the irreducible characters of $P_p\cong C_p$ are naturally labelled by the elements of $\mathbb{C}_p$. A basic application of Clifford theory allows us to extend this description to the set of linear characters of $P_{p^k}$ for any $k\in\mathbb{N}$. 
As explained in \cite[Section 2.1]{GANT}, there is a natural bijection between the set $\mathbb{C}_p^{\times k}$ and $\mathrm{Lin}(P_{p^k})$. From now on, we will denote by $\mathcal{X}(x_1,\ldots, x_k)$ the element of $\mathrm{Lin}(P_{p^k})$ corresponding to the sequence $(x_1, x_2,\ldots, x_k)\in \mathbb{C}_p^{\times k}$. Without repeating all the details concerning this bijection, we just remark that 
$\mathcal{X}(x_1,\ldots, x_k)$ is recursively defined as $$\mathcal{X}(x_1,\ldots, x_k)=\mathcal{X}(\mathcal{X}(x_1,\ldots, x_{k-1});\phi_{x_k})\in P_{p^{k-1}}\wr P_p.$$ 

Given $s=(s_1,\ldots, s_k)\in\mathbb{C}_p^{\times k}$ we let $Z(s)=\{\ell\in [1,k]\ |\ s_\ell=1\}$ and $\zeta(s)=|Z(s)|$. The following statement is an immediate consequence of the more general results proved in \cite[Proposition 2.7 and Corollary 2.8]{GANT}. 

\begin{prop}\label{prop: normAction}
Let $s, t\in \mathbb{C}_p^{\times k}$. The corresponding irreducible characters $\mathcal{X}(s)$ and $\mathcal{X}(t)$ are $N_{p^k}$-conjugate if and only if $Z(s)=Z(t)$. Moreover we have that $$\mathrm{Stab}_{N_{p^k}}(\mathcal{X}(s))\cong P_{p^k}\rtimes (C_{p-1})^{\times \zeta(s)}.$$
\end{prop}

Proposition \ref{prop: normAction} allows us to describe the set of irreducible characters of $N_{p^k}$ of degree coprime to $p$. 
Let $\mathcal{L}_k$ be a set of representatives for the orbits of $N_{p^k}$ in its action on $\mathrm{Lin}(P_{p^k})$. Another elementary application of Clifford theory and of Gallagher's theorem \cite[Theorem 6.11 and Corollary 6.17]{IBook} implies that 
$$\mathrm{Irr}_{p'}(N_{p^k})=\bigcup_{\mathcal{X}\in\mathcal{L}_k}\mathrm{Irr}(N_{p^k}\ |\ \mathcal{X}),$$
where, $\mathrm{Irr}(N_{p^k}\ |\ \mathcal{X})$ is the subset of $\mathrm{Irr}(N_{p^k})$ consisting of all those irreducible characters $\theta$ such that $[\theta_{P_{p^k}}, \mathcal{X}]\neq 0$. 

\begin{prop}\label{prop: degreesNormPpower}
Let $s\in  \mathbb{C}_p^{\times k}$ and let $\theta\in \mathrm{Irr}(N_{p^k}\ |\ \mathcal{X}(s))$. Then $\theta(1)=(p-1)^{k-\zeta(s)}$.
\end{prop}
\begin{proof}
Let $N=N_{p^k}$ and let $N_{\mathcal{X}}$ be the stabilizer of $\mathcal{X}(s)$ in $N$. 
By \cite[Theorem 6.11, and Corollaries 6.17 and 6.27]{IBook} we know that there exists $\alpha\in\mathrm{Irr}(N_{\mathcal{X}})$ and $\beta\in \mathrm{Irr}(N_{\mathcal{X}}/P_{p^k})$ such that $\alpha$ is an extension of $\mathcal{X}(s)$ and such that $\theta=(\alpha\cdot\beta)^{N}$. Since $N_{\mathcal{X}}/P_{p^k}$ is abelian we have that $\alpha(1)=\beta(1)=1$ and therefore that 
$\theta(1)=|N:N_{\mathcal{X}}|$. Proposition \ref{prop: normAction} now shows that $\theta(1)=(p-1)^{k-\zeta(s)}$. 
\end{proof}

Later in the article we will need to show that the normalizer of a Sylow $p$-subgroup of $S_n$ admits \textit{many} irreducible characters of small degree coprime to $p$. 
For this reason, here we start by counting these in the case of $N_{p^k}$. In particular we will focus on irreducible characters of degree $1$ and of degree $p-1$. 

Since the trivial character $1_{P_{p\6k}}=\mathcal{X}(1,1,\ldots ,1)$, we easily deduce the following statement as a consequence of Proposition \ref{prop: degreesNormPpower}.

\begin{cor}\label{cor: linN}
Let $k\in\mathbb{N}$. Then $\mathrm{Lin}(N_{p^k})=\mathrm{Irr}(N_{p^k}\ |\ 1_{P_{p\6k}})$. In particular we have that $|\mathrm{Lin}(N_{p^k})|=(p-1)^k$.
\end{cor}

Moreover, we denote by $\mathrm{QLin}(N_{p^k})$ the set defined by $$\mathrm{QLin}(N_{p^k})=\{\eta\in\mathrm{Irr}(N_{p^k})\ |\ \eta(1)=p-1\}.$$
The notation is chosen to hint at the small degree of these characters, in a sense, \textit{quasi-linear}.


\begin{cor}\label{cor: QlinN}
Let $k\in\mathbb{N}$. Then  $|\mathrm{QLin}(N_{p^k})|=k(p-1)^{k-1}$.
\end{cor}
\begin{proof}
For any $i\in [1,k]$ let $s(i)=(s(i)_1, s(i)_2,\ldots, s(i)_k)\in\mathbb{C}_p^{\times k}$ be the sequence defined by $s(i)_i=e^{\frac{2\pi i}{p}}$ and $s(i)_\ell=1$, for every $\ell\neq i$. Clearly $\zeta(s(i))=k-1$ for every $i\in [1,k]$. By \cite[Proposition 2.7]{GANT} we have that the $N_{p^k}$-orbits of $\mathcal{X}(s(i))$ and 
$\mathcal{X}(s(j))$ coincide if and only if $i=j$. Moreover for any $s\in \mathbb{C}_p^{\times k}$ with $\zeta(s)=k-1$, there exists an index $i\in [1,k]$ such that $\mathcal{X}(s)$ is $N_{p^k}$-conjugated to $\mathcal{X}(s(i))$. 
These observations, used together with Proposition \ref{prop: degreesNormPpower} allow us to deduce that $$\mathrm{QLin}(N_{p^k})=\bigcup_{i=1}^k\mathrm{Irr}(N_{p^k}\ |\ \mathcal{X}(s(i))).$$
Let $N_i$ be the stabilizer of $\mathcal{X}(s(i))$ in $N_{p^k}$. 
From Proposition \ref{prop: normAction}, for every $i\in [1,k]$ we have that $N_i\cong P_{p^k}\rtimes (C_{p-1})^{k-1}$. Using \cite[Corollary 6.17]{IBook} we conclude that $$|\mathrm{Irr}(N_{p^k}\ |\ \mathcal{X}(s(i)))|=|N_i : P_{p^k}|=(p-1)^{k-1},$$ and therefore that $|\mathrm{QLin}(N_{p^k})|=k(p-1)^{k-1}$, as desired. 
\end{proof}

There is a specific linear character of $P_n$ that we will need to recall quite often later in this article. For this reason we fix the following ad hoc notation.

\begin{define}\label{def: starlinear}
Let $\omega=e^{\frac{2\pi i}{p}}$ 
and let $s=(\omega, \omega, \ldots, \omega)\in \mathbb{C}_p^{\times k}$. We denote by $\mathcal{X}^\star_k$ the linear character of $P_{p^k}$ corresponding to $s$. More generally, given $n\in\mathbb{N}$ with $p$-adic expansion given by $n=p^{k_1}+\cdots+p^{k_z}+a_0$ for some $k_1\geq k_2\geq\cdots\geq k_z\geq 1$ and some $a_0\in [0,p-1]$, we denote by $\mathcal{X}^\star_{(n)}$ the linear character of $P_n$ defined by $$\mathcal{X}^\star_{(n)}=\mathcal{X}^\star_{k_1}\times\mathcal{X}^\star_{k_2}\times\cdots\times\mathcal{X}^\star_{k_z}.$$
Finally, for any $k\in\mathbb{N}$ we define the integer $m^\star(k)$ as follows: 
$$m^\star(1)=p-1,\ \text{and}\ m^\star(k)=p^k-p^{k-1}-p^{k-2},\ \text{for any}\ k\geq 2.$$ 
\end{define}

\begin{rem}
Notice that $\mathcal{X}^\star_{(n)}\neq \mathcal{X}^\star_{n}$, for any $n\in\mathbb{N}$. In fact $\mathcal{X}^\star_{(n)}\in\mathrm{Lin}(P_n)$ while $\mathcal{X}^\star_{n}\in\mathrm{Lin}(P_{p^n})$. 
We do realize that this notation might seem strange at this point, but we firmly believe it is the most convenient for the following sections of this article. 
\end{rem}

Given $\theta\in \mathrm{Irr}(P_n)$ we let $\Omega(\theta)=\{\lambda\in\mathcal{P}(n)\ |\ [(\chi^\lambda)_{P_n}, \theta]\neq 0\}$. 
From \cite[Theorem 5.7]{GL3} we deduce that the following holds. 

\begin{thm}\label{thm: GL3}
Let $p\geq 5$ be a prime and let $k\in\mathbb{N}$. 
Then $$\mathcal{B}_{p^k}(m^\star(k))\subseteq\Omega(\mathcal{X}^\star_k).$$ 

Moreover, given $n\in\mathbb{N}$ such that $n=p^{k_1}+\cdots+p^{k_z}+a_0$ is the $p$-adic expansion of $n$ for some $k_1\geq k_2\geq\cdots\geq k_z\geq 1$ and some $a_0\in [0,p-1]$, then 
$$\mathcal{B}_n(T)\subseteq \Omega(\mathcal{X}^\star_{(n)}),$$
where $T=m^\star(k_1)+m^\star(k_2)+\cdots+m^\star(k_z)+a_0$.
\end{thm}

In the next statements we make some new observations concerning the degrees of irreducible characters of $N_n$ and of $S_n$ lying above $\mathcal{X}^\star_n$.
As we will see, these results will play an important role in the proof of Theorem \ref{thm: MAIN} from the introduction.

\begin{lem}\label{lem: maxdegPpower}
Let $k\in\mathbb{N}$ and let $\theta\in\mathrm{Irr}_{p'}(N_{p^k})$. The following hold: 
\begin{itemize}
\item[(i)] $\mathrm{max}\{\eta(1)\ |\ \eta\in\mathrm{Irr}_{p'}(N_{p^k})\}=(p-1)^k$. 
\item[(ii)] $\mathrm{Irr}(N_{p^k}\ |\ \mathcal{X}^\star_k)=\{(\mathcal{X}^\star_k)^{N_{p^k}}\}$
\item[(iii)] $\theta(1)=(p-1)^k$ if and only if $\theta=(\mathcal{X}^\star_k)^{N_{p^k}}$. 
\end{itemize}
%
\end{lem}
\begin{proof}
Since the stabilizer of $\mathcal{X}^\star_k$ in $N_{p^k}$ is equal to $P_{p^k}$, statement (ii) holds by Clifford correspondence \cite[Theorem 6.11]{IBook}.
Statements (i) and (iii) easily follow from Proposition \ref{prop: degreesNormPpower}.
\end{proof}

\begin{prop}\label{prop: maxdegAP}
Let $a\in [1, p-1]$ and let $d=\max\{\chi(1)\ |\ \chi\in\mathrm{Irr}(S_a)\}$. Then 
$$\mathrm{max}\{\eta(1)\ |\ \eta\in\mathrm{Irr}_{p'}(N_{ap^k})\}=(p-1)^{ak}\cdot d.$$
\end{prop}
\begin{proof}
Let $\eta\in\mathrm{Irr}_{p'}(N_{ap^k})$. Recall that $N_{ap^k}=N_{p^k}\wr S_a=B\rtimes S_a$, where $B=(N_{p^k})^{\times a}$ is a normal subgroup of $N_{ap^k}$. Let $\theta\in\mathrm{Irr}(B)$ be such that $[\theta, \eta_B]\neq 0$. Then there exists $\ell\in [1,a]$, $\lambda=(\lambda_1,\ldots,\lambda_\ell)\in\mathcal{P}(a)$ and $\ell$ distinct characters $\theta_1,\ldots,\theta_\ell\in\mathrm{Irr}_{p'}(N_{p^k})$ such that 
$$\theta=(\theta_1)^{\times \lambda_1}\times(\theta_2)^{\times \lambda_2}\times\cdots\times(\theta_\ell)^{\times \lambda_\ell}.$$
It is easy to see that the stabilizer $H$ of $\theta$ in $N_{ap^k}$ is isomorphic to $N_{p^k}\wr S_{\lambda}$, where $S_\lambda=S_{\lambda_1}\times\cdots\times S_{\lambda_\ell}$.  Hence, using the notation for characters of wreath products introduced at the end of Section \ref{sec: 2}, we have that for every $i\in [1,\ell]$ there exists $\mu_i\in\mathcal{P}(\lambda_i)$ such that 
$$\psi=\mathcal{X}(\theta_1;\chi^{\mu_1})\times \mathcal{X}(\theta_2;\chi^{\mu_2})\times\cdots\times \mathcal{X}(\theta_\ell;\chi^{\mu_\ell})\in\mathrm{Irr}_{p'}(H),\ \text{and}\ \eta=\psi^{N_{ap^k}}.$$
It follows that $\eta(1)=\prod_i\theta_i(1)^{\lambda_i}\chi^{\mu_i}(1)|S_a : S_\lambda|$. 
From Lemma \ref{lem: maxdegPpower} we know that there exists a unique element of $\mathrm{Irr}_{p'}(N_{p^k})$ of maximal degree equal to $(p-1)^k$. This together with the fact that $\theta_i\neq\theta_j$ for all $i\neq j$ implies that 
$$\eta(1)\leq (p-1)^{k\lambda_1}(p-1)^{(k-1)(a-\lambda_1)}\cdot\prod_i\chi^{\mu_i}(1)\cdot\frac{a!}{\lambda_1!\lambda_2!\cdots\lambda_\ell!}.$$
Let $\chi\in\mathrm{Irr}(S_a)$ be an irreducible constituent of $(\chi^{\mu_1}\times\cdots\times\chi^{\mu_\ell})^{S_{a}}$. Then $$\prod_i\chi^{\mu_i}(1)\leq\chi(1)\leq d.$$ 
Moreover, since $a\leq p-1$ we observe that 
$$\frac{a!}{\lambda_1!\lambda_2!\cdots\lambda_\ell!}\leq \frac{a!}{\lambda_1!}\leq (p-1)^{a-\lambda_1}.$$
Using these inequalities we deduce that $\eta(1)\leq (p-1)^{ak}\cdot d$.
Consider now $\chi\in\mathrm{Irr}(S_a)$ with $\chi(1)=d$, and let $\zeta=(\mathcal{X}_k^\star)^{N_{p^k}}\in\mathrm{Irr}_{p'}(N_{p^k})$. We observe that $\mathcal{X}(\zeta;\chi)\in\mathrm{Irr}_{p'}(N_{ap^k})$ and that $\zeta(1)=(p-1)^{ak}\cdot d$. This concludes the proof. 
\end{proof}

\begin{cor}\label{cor: maxdegN}
Let $n=ap^k+a_0$ for some $a,a_0\in [0,p-1]$ with $a\neq 0$. Then $$\eta(1)\leq (p-1)^{ak}a!(p-1!),$$ for every $\eta\in\mathrm{Irr}_{p'}(N_n)$.
\end{cor}
\begin{proof}
We know that $N_n\cong N_{ap\6k}\times S_{a_0}$. It follows that any irreducible character $\zeta$ of degree coprime to $p$ of $N_n$ is of the form $\zeta=\theta\times\chi\6\gamma$, for some $\theta\in\mathrm{Irr}_{p'}(N_{ap\6k})$ and some $\gamma\in\mathcal{P}(a_0)$. From Proposition \ref{prop: maxdegAP} we know that $\theta(1)\leq (p-1)^{ak}\cdot d$, where $d$ is the maximal degree of an irreducible character of $S_a$. Since $d\leq a!$ and since $a_0\leq p-1$ we deduce that 
$$\zeta(1)\leq (p-1)^{ak}\cdot d\cdot\chi\6\gamma(1)\leq (p-1)^{ak}a!(p-1!).$$
\end{proof}

We now aim to provide lower bounds for the degrees of the irreducible characters of $S_n$ lying above $\mathcal{X}^\star_{(n)}$. This is done in Lemmas \ref{lem: 121} and \ref{lem: 11} below. 
Before doing this, we start by recalling a few facts about the minimal degrees of irreducible characters of symmetric groups. 
The first one is very well-known and we refer the reader to \cite[Result 1]{Rasala} for a proof. 

\begin{lem}\label{lem: Rasala1}
Let $n\geq 5$ be a natural number and let $\chi\in\mathrm{Irr}(S_n)$. If $\chi(1)>1$ then $\chi(1)\geq n-1$. 
\end{lem}

The following statement is formulated specifically for the purposes of this article and it is a consequence of the more general results described in \cite{Rasala}.

\begin{lem}\label{lem: Rasala2}
Let $n\in\mathbb{N}$, let $\lambda\in\mathcal{P}(n)$ and let $\chi=\chi^\lambda\in\mathrm{Irr}(S_n)$.
\begin{itemize}
\item[(a)] If $n\geq 9$ and $\lambda\in\mathcal{B}_n(n-2)$, then $\chi(1)\geq \frac{1}{2}n(n-3)$.
\item[(b)] If $n\geq 15$ and $\lambda\in\mathcal{B}_n(n-3)$, then $\chi(1)\geq \frac{1}{6}n(n-1)(n-5)$.
\end{itemize}
\end{lem}
\begin{proof}
Let $n=2m+\varepsilon$, for some $\varepsilon\in [0,1]$, and let $t\in [1,m]$.
From \cite[Theorem F]{Rasala} we have that for every $\lambda\in\mathcal{B}_n(n-t)$ there exists $\mu\in\{(n-t,t), (n-m,m)\}$ such that $\chi^\lambda(1)\geq \chi^\mu(1)$. 
Since $\chi^{(n-2,2)}(1)=\frac{1}{2}n(n-3)$ statement $(a)$ now follows from \cite[Result 2]{Rasala}. Similarly, since $\chi^{(n-3,3)}(1)=\frac{1}{6}n(n-1)(n-5)$ statement $(b)$ follows directly from \cite[Result 3]{Rasala}.
\end{proof}

We record here a useful lemma specific to the prime $p=3$. 

\begin{lem}\label{lem: lemma333}
Let $n=a3^k+r$ for some $a\in [1,2]$ and $r\in [0,2]$. If $\lambda\in\mathcal{B}_n(n-a)$ then $\chi^\lambda(1)\geq \eta(1)$, for all $\eta\in\mathrm{Irr}_{3'}(N_n)$. 
\end{lem} 
\begin{proof}
From Proposition \ref{prop: maxdegAP} we know that $\eta(1)\leq 2^{ak}$, for all $\eta\in\mathrm{Irr}_{3'}(N_n)$. The statement is now an immediate consequence of Lemma \ref{lem: Rasala2}.
\end{proof}

As promised, we are now ready to study the degree of those irreducible characters of $S_n$ lying above $\mathcal{X}^\star_{(n)}$. This analysis will be crucial to prove Theorem 1.1 for $p\geq 5$. 

\begin{lem}\label{lem: 121}
Let $p$ be a prime number and let $k\geq 1$. 
Let $\chi\in\mathrm{Irr}(S_{p^k})$ be such that $[\chi_{P_{p^k}}, \mathcal{X}^\star_k]\neq 0$, then $\chi(1)\geq (p-1)^{p^{k-1}}.$
\end{lem}
\begin{proof}
We proceed by induction on $k$. If $k=1$ the result clearly holds since $\chi\notin\mathrm{Lin}(S_n)$ and therefore $\chi(1)\geq p-1$ by Lemma \ref{lem: Rasala1} below. 
Suppose now that $k\geq 2$. 
Let $B$ be the normal subgroup of $P_{p\6k}$ such that $B\cong (P_{p\6{k-1}})\6{\times p}$ and such that $P_{p\6k}=B\rtimes P_p\leq S_{p\6k}$.
Since $\mathcal{X}^\star_k=\mathcal{X}(\mathcal{X}^\star_{k-1}; \phi_{\omega})$, we have that 
$$[\chi_B, (\mathcal{X}^\star_{k-1})\6{\times p}]\neq 0.$$ Hence
 there exists $\chi_1,\ldots, \chi_{p}\in\mathrm{Irr}(S_{p^{k-1}})$ such that 
$[\chi_Y, \chi_1\times\cdots\times\chi_{p}]\neq 0$ and such that $[(\chi_j)_{P_{p^{k-1}}}, \mathcal{X}^\star_{k-1}]\neq 0$, for all $j\in [1,p]$. 
Here $Y$ denotes the Young subgroup $(S_{p^{k-1}})^{\times p}$ in $S_{p^k}$, chosen such that $B\leq Y$. The inductive hypothesis implies that $\chi_j(1)\geq (p-1)^{p^{k-2}}$ and hence we conclude that $$\chi(1)\geq \prod_{j=1}^p\chi_j(1)\geq \big((p-1)^{p^{k-2}}\big)^p=(p-1)^{p^{k-1}}.$$
\end{proof}

\begin{lem}\label{lem: 11}
Let $p\geq 5$ be a prime number, let $a,a_0\in [0,p-1]$ with $a\neq 0$ and let $k\geq 2$. 
If $n=ap^k+a_0$ and $\chi\in\mathrm{Irr}(S_n)$ is such that $[\chi_{P_n}, \mathcal{X}^\star_{(n)}]\neq 0$, then 
$$\chi(1)\geq (p-1)^{ak}a!(p-1!).$$
\end{lem}
\begin{proof}
Let $Y=(S_{p^k})^{\times a}\times S_{a_0}\leq S_n$ and observe that $\mathcal{X}^\star_{(n)}=(\mathcal{X}^\star_k)^{\times a}$.
Since $[\chi_{P_n}, \mathcal{X}^\star_{(n)}]\neq 0$, there exist $\chi^1,\ldots, \chi^a\in\mathrm{Irr}(S_{p^k})$ and $\psi\in \mathrm{Irr}(S_{a_0})$ such that 
$[\chi_Y, \chi^1\times\cdots\times\chi^a\times \psi]\neq 0$ and such that $[(\chi^j)_{P_{p^{k}}}, \mathcal{X}^\star_{k}]\neq 0$, for all $j\in [1,a]$.
Lemma \ref{lem: 121} shows that $\chi^j(1)\geq (p-1)^{p^{k-1}}$ for all $j\in [1,a]$ and therefore we have that $\chi(1)\geq (p-1)^{ap^{k-1}}$. 
The statement now follows by observing that 
 $$(p-1)^{ap^{k-1}}\geq (p-1)^{ak}(p-1)^{a-1}(p-1)^{p-2}\geq (p-1)^{ak}a!(p-1!).$$
 The first inequality is a direct consequence of Lemma \ref{lem: analisi1} (see appendix). The second inequality is obvious because $a\leq p-1$.
 \end{proof}

\subsection{More on symmetric groups}

Roughly speaking, Theorem \ref{thm: GL3} shows that a partition $\lambda$ that is neither \textit{too wide} or \textit{too tall} corresponds to an irreducible character $\chi^\lambda$ whose restriction to $P_n$ picks up $\mathcal{X}^\star_{(n)}$ as an irreducible constituent. This in turn implies that $\chi^\lambda(1)$ is \textit{big}, by Lemmas \ref{lem: 121} and \ref{lem: 11}. For these reasons it is crucial for us to find combinatorial conditions to control how wide and how tall a partition is. This is the first goal of this section.

\begin{define}\label{def: nlambda}
Let $\lambda\in\mathcal{P}(n)$ and let $s$ be a natural number. Let $Q_s(\lambda)=(\lambda^{(0)},\ldots, \lambda^{(s-1)})$ be the $s$-quotient of $\lambda$. We let $N_{s}(\lambda)$ be the integer defined as follows: 
$$N_{s}(\lambda)=\mathrm{max}\{(\lambda^{(j)})_1, \ell(\lambda^{(j)})\ |\ j\in [0,s-1]\}.$$
In other words, $N_s(\lambda)$ is the largest number among all first parts and all lengths (i.e. number of parts) of the $s$ partitions involved in the $s$-quotient of $\lambda$. 
%
\end{define}

\begin{rem}\label{rem: Nslambda}
Notice that given $s\in\mathbb{N}$, the integer $N_s(\lambda)$ depends only on the partition $\lambda$ and does not depend on the choice of the $\beta$-set for $\lambda$. This fact can be easily deduced from Remark \ref{rem: quotient}.
\end{rem}

We pause to give a few examples of the new invariant introduced in Definition \ref{def: nlambda}.

\begin{example}
Let $\lambda=(7,5,5,3,3)\in\mathcal{P}(23)$ and let $s=5$. The set $$X=\{16,13,12,9,8,4,3,2,1,0\}$$ is a $\beta$-set for $\lambda$. In fact we have that $X=X_\lambda^{+5}$, where $X_\lambda=\{11,8,7,4,3\}$.
Considering the corresponding $5$-Abacus configuration we see that $C_5(\lambda)=(4,1,1,1,1)$ and that $$Q_5(\lambda)=(\emptyset,(2),(1),\emptyset,\emptyset).$$ It follows that $N_5(\lambda)=2$. 

On the other hand, given $\mu=(7,7,4,3,2,2,1\6{8})\in\mathcal{P}(33)$ we have that the set 
$$Y=\{23,22,18,16,14,13,11,10,9,8,7,6,5,4,2,1,0\},$$
is a $\beta$-set for $\mu$. It follows that 
$C_5(\mu)=(2,1)$ and that $$Q_5(\mu)=(\emptyset,\emptyset,(2), (1\6{4}),\emptyset).$$
We conclude that $N_5(\mu)=4$. 
It is important to notice that $$\lambda_1=7\leq 8+5\cdot 2=|C_5(\lambda)|+5\cdot N_5(\lambda),\
\text{and that}\ \ell(\mu)=14\leq 3+5\cdot 4=|C_5(\mu)|+5\cdot N_5(\mu).$$
The two equations displayed above are instances of a more general fact, that is fully explained by Proposition \ref{prop: decisiva} below. 
\end{example}

The following proposition will be used frequently in our proof of Theorem \ref{thm: MAIN} and it is proved using the combinatorics of partitions and $\beta$-sets, as introduced it in Section \ref{sec: beta}

\begin{prop}\label{prop: decisiva}
Let $a,n,r,s\in\mathbb{N}$ be such that $n=as+r$ and $r<s$. Let
$\lambda\in\mathcal{P}(n)$ be such that $C_s(\lambda)\in\mathcal{P}(r)$. Then $\lambda\in\mathcal{B}_n(r+sN_s(\lambda))$.
\end{prop}
\begin{proof}
Let $X=(X_\gamma)\6{+as}$ be the $\beta$-set for $\gamma$ corresponding to the $s$-abacus configuration for $\gamma$ having first vacant position labelled by $as$ (position $a$ on runner $0$). 
We remark that $X$ (as well as all the other $\beta$-sets considered in this proof) may have size not divisible by $s$. This does not affect our arguments, as explained in Remarks \ref{rem: quotient} and \ref{rem: Nslambda}.
Since $\gamma\in\mathcal{P}(r)$ is an $s$-core and since $r<s$ we have that $x<(a+1)s$ for every $x\in X$. Therefore we deduce that 
$$X^q\in\{\{0,1,2,\ldots, a-1\}, \{0,1,2,\ldots, a\}\}, \text{for  every}\ q\in [0,s-1].$$ 
Let $Q_s(\lambda)=(\lambda^{0},\ldots, \lambda^{(s-1)})$ be the $s$-quotient of $\lambda$ and let $\lambda^{(q)}=(d_1^q,d_2^q,\ldots, d_{t_q}^q)$, for every $q\in [0,s-1]$ (here $t_q$ is some integer smaller than or equal to $a$). 
Regarding $d_x^q=0$ for every $x>t_q$ we let 
$$Y^q= \begin{cases}
		\{(a-x)+d_x^q\ |\ x\in [1,a]\} & \text{if } X^q=\{0,1,2,\ldots, a-1\},\\
		
		\{(a-x)+d_{x+1}^q\ |\ x\in [0,a]\} & \text{if }  X^q=\{0,1,2,\ldots, a-1,a\}.
\end{cases}$$
It is easy to see that $P(Y^q)=\lambda^{(q)}$ for every $q\in [0,s-1]$, and that $Y=\{s\cdot y+q\ |\ y\in Y^q\}$ is a $\beta$-set for $\lambda$, i.e. $P(Y)=\lambda$. 
In fact, combinatorially speaking, $Y$ is obtained from $X$ by sliding down beads on runner $q$ according to the partition $\lambda^{(q)}$. 
Let $M=\mathrm{max}(Y)$ and let $q\in [0,s-1]$ be such that $M=xs+q$ (i.e. the $s$-abacus configuration corresponding to $Y$ has a bead in position $x$ on runner $q$). 
From the definition of $Y$ we have that $$M\in\{(a-1+d_1^q)s+q, (a+d_1^q)s+q\}.$$ 
Moreover, since $\lambda=P(Y)$, we know that $\lambda_1=M-(|Y|-1)$.  
Since $|Y|=|X|=as+\ell(\gamma)$ we have that 

$$\lambda_1=\begin{cases} sd_1^q-s-\ell(\gamma)+q+1 & \text{if } X^q=\{0,1,2,\ldots, a-1\},\\
		
		sd_1^q-\ell(\gamma)+q+1  & \text{if }  X^q=\{0,1,2,\ldots, a-1,a\}.
\end{cases}$$

If $a\notin X^q$ then using that $q\leq s-1$ and that $\ell(\gamma)\geq 0$ we have that 
$$\lambda_1=sd_1^q-s-\ell(\gamma)+q+1\leq sd_1^q+r.$$ 
On the other hand if $a\in X^q$ then $q\leq r$ (because $\gamma\in \mathcal{P}(r))$ and therefore we have again that 
$\lambda_1=sd_1^q-\ell(\gamma)+q+1\leq sd_1^q+r$. Since $d^q_1\leq N_s(\lambda)$
We conclude that $$\lambda_1\leq sd_1^q+r\leq sN_s(\lambda)+r,$$ as desired. 

A completely similar argument shows that $\ell(\lambda)\leq sN_s(\lambda)+r$. 
\end{proof}

We note that the assumption $r<s$ in the statement of Proposition \ref{prop: decisiva} may be superfluous. However, we have included it because, in this article, we only apply Proposition \ref{prop: decisiva} in contexts where this condition is certainly satisfied.

We fix some new notation that will be used frequently later on in the paper. 

\begin{nota}\label{not: Deltas}
Let $p$ be a prime and let $n=ap^k+r$ for some $r<p^k$ and some $a\in [1,p-1]$. Given any $\gamma\in\mathcal{P}_{p'}(r)$ we let $$\mathcal{P}(n\ |\ \gamma)=\{\lambda\in\mathcal{P}(n)\ |\ C_{p^k}(\lambda)=\gamma\}.$$
We observe that $\mathcal{P}(n\ |\ \gamma)\subseteq\mathcal{P}_{p'}(n)$ by Proposition \ref{prop: pdivisibility}, and moreover
 $\lambda\in\mathcal{P}(n\ |\ \gamma)$ if and only if $\chi^\lambda\in\mathrm{Irr}(S_n\ |\ \gamma)$. 

Given $\lambda\in \mathcal{P}(n\ |\ \gamma)$ we have that $Q_{p^k}(\lambda)$ is a sequence of $p^k$ partitions whose sizes sum up to $a$. With this in mind, for every $x\in [1,a]$ we denote by $\Delta_x$ the subset of $\mathcal{P}(n\ |\ \gamma)$ defined by
$$\Delta_x=\{\lambda\in\mathcal{P}(n\ |\ \gamma)\ |\ N_{p\6k}(\lambda)=x\}.$$

\smallskip

Observe that $\Delta_x$ is non empty for every $x\in [1,a]$, and that $$\mathcal{P}(n\ |\ \gamma)=\bigcup_{x\in [1,a]}\Delta_x,$$
where the above union is clearly a disjoint union. 
We remark that the notation $\Delta_x$ does not record the dependence on $n$ and $\gamma$ because these are fixed at the start and clear from the context. Similarly, when later in the article we will make use of the sets $\Delta_1,\Delta_2,\ldots, \Delta_a$ the variables $n$ and $\gamma$ will always be fixed and clear from the context.  

We conclude by observing that $\Delta_x\subseteq \mathcal{B}_n(r+xp^k)=\mathcal{B}_n(n-(a-x)p^k)$. This is an immediate consequence of Proposition \ref{prop: decisiva}.
\end{nota}

Using the notation just introduced, we compute the sizes of the sets $\Delta_x$, for a few specific values of $x$. 

\begin{lem}\label{lem: sizeDelta}
Let $p$ be a prime and let $n=ap^k+r$ for some $r<p^k$ and some $a\in [2,p-1]$. Then 
$|\Delta_a|=2p\6k$. Moreover, if $a\geq 3$ then 
$$|\Delta_{a-1}|=\begin{cases}
		2p\6{2k}-p\6k & \text{if } a=3,\\
		
		2p\6{2k} & \text{if }  a\geq 4.
\end{cases}
$$
Finally, assuming that $a\geq 6$ then we have 
$$|\Delta_{a-2}|= p\6{3k}+3p\6{2k}.$$
\end{lem}
\begin{proof}
The first statement is easily seen because $\lambda\in\Delta_a$ if and only if $C_{p^k}(\lambda)=\gamma$ and $N_{p^k}(\lambda)=a$. This in turn is equivalent to have $C_{p^k}(\lambda)=\gamma$ and $Q_{p^k}(\lambda)=(\lambda^{(0)}, \lambda^{(1)}, \ldots, \lambda^{(p^k-1)})$ satisfying the following conditions:
there exists a unique $j\in [0,p^k-1]$ such that $\lambda^{(i)}=\emptyset$ for every $i\neq j$ and $\lambda^{(j)}\in \{(a), (1^a)\}$. Since $a\geq 2$ we conclude that $|\Delta_a|=2p\6k$.

To prove the second statement we observe that $\Delta_{a-1}$ consists of all those partitions $\lambda$ such that $C_{p^k}(\lambda)=\gamma$ and such that $Q_{p^k}(\lambda)=(\lambda^{(0)},\lambda^{(1)},\ldots,\lambda^{(p^k-1)})$ is a sequence of $p^k$ partitions that is obtained by permuting the components of one of the following sequences: 
$$(\nu, \emptyset,\emptyset, \ldots, \emptyset)\ \text{or}\ (\rho,(1),\emptyset, \ldots, \emptyset),$$
where $\nu\in\{(a-1,1), (2,1^{a-2})\}$ and $\rho\in\{(a-1), (1^{a-1})\}$. 
Observing that when $a=3$ we have $(a-1,1)=(2,1)=(2,1\6{a-2})$, we conclude that 
$$|\Delta_{a-1}|=\begin{cases}
		2p\6{2k}-p\6k & \text{if } a=3,\\
		
		2p\6{2k} & \text{if }  a\geq 4,
\end{cases}
$$
as desired. 
The third statement, concerning the size of $|\Delta_{a-2}|$ is proved with a completely similar argument. The details are therefore omitted. 
\end{proof}

\section{A proof of Theorem \ref{thm: MAIN}}\label{sec: theorem B}

The goal of this section is to show that for any fixed prime number $p\geq 3$, there exists a bijection 
$$\varepsilon_n:\mathrm{Irr}_{p'}(S_n)\rightarrow\mathrm{Irr}_{p'}(N_n),$$ such that $\varepsilon_n(\chi)(1)\leq\chi(1)$, for every $\chi\in\mathrm{Irr}_{p'}(S_n)$. 

In order to do this we will proceed by induction on the $p$-adic length of the natural number $n$. The following Proposition \ref{prop: strategy} allows us to reduce the problem and gives an illustration of our strategy to construct the desired bijection $\varepsilon_n$.

Before stating it we introduce the following piece of notation. Let $n=ap\6k+r$ for some $r<p\6k$ and some $a\in [1,p-1]$. From Section \ref{subsec: normalisers} we know that $N_n=N_{ap\6k}\times N_r$. For any $\eta\in\mathrm{Irr}_{p'}(N_r)$ we let $\mathrm{Irr}_{p'}(N_n\ |\ \eta)$ be the subset of $\mathrm{Irr}_{p'}(N_n)$ defined by 
$$\mathrm{Irr}_{p'}(N_n\ |\ \eta)=\{\theta\times\eta\ |\ \theta\in\mathrm{Irr}_{p'}(N_{ap\6k})\}.$$
It is clear that $\mathrm{Irr}_{p'}(N_n)$ is equal to the disjoint union of the sets $\mathrm{Irr}_{p'}(N_n\ |\ \eta)$, where $\eta$ runs among all elements of $\mathrm{Irr}_{p'}(N_r)$.

\begin{rem}\label{rem: McKay2}
Given any $\gamma\in\mathcal{P}_{p'}(r)$ and any $\eta\in\mathrm{Irr}_{p'}(N_r)$, we have that $$|\mathrm{Irr}(S_n\ |\ \gamma)|=|\mathrm{Irr}_{p'}(N_n\ |\ \eta)|.$$
This is an immediate consequence of Lemma \ref{lem: McKay}, since $|\mathrm{Irr}_{p'}(N_n\ |\ \eta)|=|\mathrm{Irr}_{p'}(N_{ap^k})|$. 
\end{rem}

\begin{prop}\label{prop: strategy}
Let $n=ap\6k+r$ for some $r<p\6k$ and some $a\in [1,p-1]$. Suppose that the following hold. 
\begin{itemize}
\item[(a)] There exists a bijection $\varepsilon_r:\mathrm{Irr}_{p'}(S_r)\rightarrow\mathrm{Irr}_{p'}(N_r),$ such that $\varepsilon_r(\chi)(1)\leq\chi(1)$, for every $\chi\in\mathrm{Irr}_{p'}(S_r)$. 

\item[(b)] For every $\gamma\in\mathcal{P}_{p'}(r)$ there exists a bijection $\varepsilon_\gamma: \mathrm{Irr}(S_n\ |\ \gamma)\rightarrow\mathrm{Irr}_{p'}(N_n\ |\ \varepsilon_r(\chi\6\gamma))$, such that $\varepsilon_\gamma(\chi)(1)\leq\chi(1)$ for all $\chi\in \mathrm{Irr}(S_n\ |\ \gamma)$. 
\end{itemize}
Then there exists a bijection 
$\varepsilon_n:\mathrm{Irr}_{p'}(S_n)\rightarrow\mathrm{Irr}_{p'}(N_n),$ such that $\varepsilon_n(\chi)(1)\leq\chi(1)$, for every $\chi\in\mathrm{Irr}_{p'}(S_n)$. 
\end{prop}
\begin{proof}
From Hypothesis (a) we have that $$\mathrm{Irr}_{p'}(N_n)=\bigcup_{\gamma\in\mathcal{P}_{p'}(r)}\mathrm{Irr}_{p'}(N_n\ |\ \varepsilon_r(\chi\6\gamma)).$$ Similarly, the discussion after Proposition \ref{prop: pdivisibility} implies that $$\mathrm{Irr}_{p'}(S_n)=\bigcup_{\gamma\in\mathcal{P}_{p'}(r)}\mathrm{Irr}(S_n\ |\ \gamma).$$
Both the above unions are clearly disjoint. 
Hence, setting $\varepsilon_n(\chi)=\varepsilon_\gamma(\chi)$ for every $\chi\in\mathrm{Irr}(S_n\ |\ \gamma)$ and for every $\gamma\in\mathcal{P}_{p'}(r)$, defines a bijection between $\mathrm{Irr}_{p'}(S_n)$ and $\mathrm{Irr}_{p'}(N_n)$ that satysfies the desired inequality between degrees of corresponding characters.  
\end{proof}

In light of Proposition \ref{prop: strategy} we will now use the inductive hypthesis to assume the validity of Hypothesis (a) and we will focus on proving Hypothesis (b) for every $\gamma\in\mathcal{P}_{p'}(r)$. 

We start by dealing with a special case. 

\begin{prop}\label{prop: auguale1}
Let $r,k\in\mathbb{N}$ be such that $p\6k>r$. Suppose that $n=p\6k+r$ and that there exists a bijection $\varepsilon_r:\mathrm{Irr}_{p'}(S_r)\rightarrow\mathrm{Irr}_{p'}(N_r),$ such that $\varepsilon_r(\chi)(1)\leq\chi(1)$, for every $\chi\in\mathrm{Irr}_{p'}(S_r)$. 
Then there exists a bijection $$\varepsilon_n: \mathrm{Irr}_{p'}(S_n)\rightarrow \mathrm{Irr}_{p'}(N_n),$$
such that $\varepsilon_n(\chi)(1)\leq\chi(1)$, for all $\chi\in \mathrm{Irr}_{p'}(S_n).$
\end{prop}
\begin{proof}
We observe that Hypothesis (a) of Proposition \ref{prop: strategy} is satisfied. Hence we just have to show that Hypothesis (b) of Proposition \ref{prop: strategy} holds as well. To this end, we fix $\gamma\in\mathcal{P}_{p'}(r)$ and we observe that Lemma \ref{lem: bess} guarantees that $\mathcal{P}(n\ |\ \gamma)=\{\lambda_0,\lambda_1,\ldots,\lambda_{p\6k-1}\},$
where for every $i\in [0,p\6k-1]$ we have that $[\chi^{(p\6k-i,1^i)}\times\chi^\gamma, (\chi^{\lambda_i})_{S_{p\6k}\times S_r}]\neq 0$. In particular, for every $i\in [1,p\6k-2]$ and for every $\eta\in\mathrm{Irr}_{p'}(N_n\ |\ \varepsilon_r(\chi\6\gamma))$  we have that 
$$\chi^{\lambda_i}(1)\geq \chi^{(x-i,1^i)}(1)\cdot\chi^\gamma(1)\geq (p\6k-1)\cdot\chi\6\gamma(1)\geq (p-1)\6k\varepsilon_r(\chi\6\gamma)(1)\geq\eta(1).$$
Here the second inequality follows from Lemma \ref{lem: Rasala1}, the third one is a direct consequence of our hypotheses and the fourth one follows from Lemma \ref{lem: maxdegPpower}.
We just proved that, with the sole exception of $\chi^{\lambda_0}$ and of $\chi^{\lambda_{p\6k-1}}$, every other irreducible characters in $\mathrm{Irr}(S_n\ |\ \gamma)$ has degree larger than the degree of any irreducible character in $\mathrm{Irr}_{p'}(N_n\ |\ \varepsilon_r(\chi\6\gamma))$. To conclude we use Lemma \ref{lem: LR} to observe that for any $i\in \{0,p\6k-1\}$ we have 
$$\chi\6{\lambda_i}(1)\geq \chi\6\gamma(1)\geq \varepsilon_r(\chi\6\gamma)(1).$$
Let $N_0$ be the subset of $\mathrm{Irr}_{p'}(N_n\ |\ \varepsilon_r(\chi\6\gamma))$ consisting of those irreducible characters whose degree is equal to $\varepsilon_r(\chi\6\gamma)(1)$. 
From Corollary \ref{cor: linN} we know that $|N_0|=(p-1)\6k\geq 2$. 
Moreover, Remark \ref{rem: McKay2} shows that $$|\mathrm{Irr}(S_n\ |\ \gamma)|=|\mathrm{Irr}_{p'}(N_n\ |\ \varepsilon_r(\chi\6\gamma))|.$$
It follows that there exists a bijection 
$\varepsilon_\gamma: \mathrm{Irr}(S_n\ |\ \gamma)\rightarrow\mathrm{Irr}_{p'}(N_n\ |\ \varepsilon_r(\chi\6\gamma))$, such that $\varepsilon_\gamma(\{\chi\6{\lambda_0},\chi\6{\lambda_{p\6k-1}}\})\subseteq N_0$. 
The observations above guarantee that $\varepsilon_\gamma(\chi)(1)\leq\chi(1)$ for all $\chi\in \mathrm{Irr}(S_n\ |\ \gamma)$. This, together with Proposition \ref{prop: strategy}, concludes the proof.
\end{proof}

We pause to fix the notation for a collection of subsets of $\mathrm{Irr}_{p'}(N_n)$ consisting of irreducible characters of \textit{small} degree. We will repeatedly use these sets in the proof of several statements in Sections \ref{sec: theorem B} and \ref{sec: appendix}. 
We first recall that $N_{ap^k}=N_{p^k}\wr S_a$ and we refer the reader to Section \ref{sec: 2.2} for the notation we use for characters of wreath products.

\begin{nota}\label{not: subsets}
For any $k\in\mathbb{N}$ and any $a\in [2,p-1]$ let us consider the subsets $X_0(k), X(k)$, $Y_0(k)$ and $R(k)$ of $\mathrm{Irr}_{p'}(N_{ap^k})$ defined as follows: 
$$X_0(k)=\{\mathcal{X}(\phi; \chi^\nu)\ |\ \phi\in\mathrm{Lin}(N_{p^k}),\ \nu\in\{(a)\}^{\circ}\},$$
$$X(k)=\{\mathcal{X}(\phi; \chi^\nu)\ |\ \phi\in\mathrm{Lin}(N_{p^k}),\ \nu\in\{(a), (a-1,1)\}^{\circ}\},$$
$$Y_0(k)=\{(\mathcal{X}(\phi; \chi^\nu)\times\eta)^{N_{ap^k}}\ |\ \phi,\eta\in\mathrm{Lin}(N_{p^k}),\ \nu\in\{(a-1)\}^{\circ}\},$$
$$R(k)=\{(\mathcal{X}(\phi; \chi^\nu)\times\eta)^{N_{ap^k}}\ |\ \phi\in\mathrm{Lin}(N_{p^k}),\ \eta\in\mathrm{QLin}(N_{p^k}),\ \nu\in\{(a-1)\}^{\circ}\}.$$

\smallskip

Moreover, for $p\geq 5$ and $a\in [3,p-1]$, we define $Y(k), Z(k)\subseteq\mathrm{Irr}_{p'}(N_{ap^k})$ as follows:
$$Y(k)=\{(\mathcal{X}(\phi; \chi^\nu)\times\eta)^{N_{ap^k}}\ |\ \phi,\eta\in\mathrm{Lin}(N_{p^k}),\ \nu\in\{(a-1), (a-2,1)\}^{\circ}\}.$$
$$Z(k)=\{(\mathcal{X}(\phi; \chi^\nu)\times\eta_1\times\eta_2)^{N_{ap^k}}\ |\ \phi,\eta_1,\eta_2\in\mathrm{Lin}(N_{p^k}),\ \nu\in\{(a-2)\}^{\circ}\}.$$

\smallskip

We collect here a few remarks and observations concerning these five sets. 

\smallskip

\noindent\textbf{(a)} From Corollary \ref{cor: linN}, we deduce that the set $X_0(k)$ consists of $2(p-1)^k$ linear characters. Moreover, $X_0(k)\subseteq X(k)$ and every character in $X(k)\smallsetminus X_0(k)$ has degree equal to $a-1$. 

\smallskip

\noindent\textbf{(b)} In the definition of the sets $Y_0(k)$ and $Y(k)$ the linear characters $\phi$ and $\eta$ are chosen to be distinct, so that $\mathcal{X}(\phi; \chi^\nu)\times\eta$ is an irreducible character of $N_{(a-1)p^k}\times N_{p^k}\leq N_{ap^k}$, that induces irreducibly to $N_{ap^k}$. 
In particular we have that $\zeta(1)=a$ for every $\zeta\in Y_0(k)$, and that $\theta(1)\leq a(a-2)$ for every $\theta\in Y(k)$.

\smallskip

\noindent\textbf{(c)} Similarly, in the the definition of the set $R(k)$ we have that $\mathcal{X}(\phi; \chi^\nu)\times\eta$ is an irreducible character of degree $p-1$ of $N_{(a-1)p^k}\times N_{p^k}\leq N_{ap^k}$, that induces irreducibly to $N_{ap^k}$. 
It follows that $\zeta(1)=a(p-1)$ for every $\zeta\in R(k)$.

\smallskip

\noindent\textbf{(d)} In the definition of the set $Z(k)$ the linear characters $\phi$, $\eta_1$ and $\eta_2$ are chosen to be pairwise distinct, so that $\mathcal{X}(\phi; \chi^\nu)\times\eta_1\times\eta_2$ is an irreducible character of $N_{(a-2)p^k}\times N_{p^k}\times N_{p^k}\leq N_{ap^k}$, that induces irreducibly to $N_{ap^k}$. 
Again, we have that $\theta(1)\leq a(a-1)$, for every $\theta\in Z(k)$. 

\smallskip

\noindent\textbf{(e)} From $(a), (b)$ and $(c)$ for every $\theta\in X(k)\cup Y(k)\cup R(k)\cup Z(k)$ we have that $$\theta(1)\leq a(p-1)\leq ap^k-1.$$

\smallskip

Using the discussion in Section \ref{subsec: normalisers}, it is straightforward to obtain lower bounds for the sizes of the sets introduced above. (In most cases, we actually compute their exact sizes.)
The calculations below are carried out for specific values of the parameter $a$, chosen to match the cases that arise in the proofs of Theorems \ref{thm: base}, \ref{thm: B2}, and Proposition \ref{prop: base1}.

If $a=3$ then $(a-1,1)=(2,1\6{a-2})$ and hence 
$$|X(k)|=3(p-1)^k,\ |Y(k)|=2(p-1)^k((p-1)^k-1),\ |Z(k)|={(p-1)^k\choose 3}.$$

On the other hand, if $a\geq 4$ then
$$|X(k)|=4(p-1)^k,\ |Y(k)|\geq 3(p-1)^k((p-1)^k-1),\ |Z(k)|\geq 6\cdot{(p-1)^k\choose 3}.$$

Moreover, we have that 

$$|Y_0(k)|= \begin{cases}
		(p-1)^k\choose 2 & \text{if } a=2,\\
		
		\\
		
		2(p-1)^k((p-1)^k-1) & \text{if } a\geq 3.
	\end{cases}$$
	
	\smallskip
	
Finally, from Corollary \ref{cor: QlinN} we deduce that

$$|R(k)|= \begin{cases}
		k\cdot (p-1)^{2k-1} & \text{if } a=2,\\
		
		\\
		
		2k\cdot (p-1)^{2k-1} & \text{if } a\geq 3.
	\end{cases}$$
\end{nota}

We are now ready to treat the case where $n$ has $p$-adic length equal to $1$. This is the base case of our inductive argument. 

\begin{thm}\label{thm: base}
Let $n=ap^k+r$ for some $a\in [1,p-1]$ and some $r\in [0,p-1]$. There exists a bijection 
$\varepsilon_n:\mathrm{Irr}_{p'}(S_n)\rightarrow\mathrm{Irr}_{p'}(N_n),$ such that $\varepsilon_n(\chi)(1)\leq\chi(1)$, for every $\chi\in\mathrm{Irr}_{p'}(S_n)$. 
\end{thm}
\begin{proof}
The case $k=1$ is surprisingly difficult, it is treated separately in Section \ref{sec: appendix} and it is completed in Proposition \ref{prop: base1}.
Let us now assume that $k\geq 2$.
Since $r<p$ we have that $N_r=S_r$ and that $\mathcal{P}_{p'}(r)=\mathcal{P}(r)$. In particular Hypothesis (a) of Proposition \ref{prop: strategy} is satisfied by taking $\varepsilon_r$ as the identity map. Thus, in order to prove the present theorem we just need to show that Hypothesis (b) of  Proposition \ref{prop: strategy} holds. 
To this end, we fix $\gamma\in\mathcal{P}(r)$ and we aim at showing the existence of a bijection $\varepsilon_{\gamma}$ between
$\mathrm{Irr}(S_n\ |\ \gamma)$ and $\mathrm{Irr}_{p'}(N_n\ |\ \gamma)$ such that $\varepsilon_{\gamma}(\chi)(1)\leq\chi(1)$ for every $\chi\in \mathrm{Irr}(S_n\ |\ \gamma)$.
Here, and for the rest of the proof, we relax the notation by writing $\mathrm{Irr}_{p'}(N_n\ |\ \gamma)$ instead of $\mathrm{Irr}_{p'}(N_n\ |\ \chi\6\gamma)$.
Before moving on, recall
 that $|\mathrm{Irr}(S_n\ |\ \gamma)|=|\mathrm{Irr}_{p'}(N_n\ |\ \gamma)|$, by Remark \ref{rem: McKay2}.

Notice that Proposition \ref{prop: auguale1} shows that the statement holds when $a=1$. Hence we can now work under the assumption that $a\geq 2$. 
For any $\lambda\in\mathcal{P}(n)$ consider its $p^k$-quotient $Q_{p^k}(\lambda)=(\lambda^{(0)}, \ldots, \lambda^{(p^k-1)})$. Recalling Definition \ref{def: nlambda}, we denote by $N_{p\6k}(\lambda)$  the size of the largest number among all first parts and all lengths (i.e. number of parts) of the $p^k-1$ partitions involved in $Q_{p^k}(\lambda)$.
With this is mind, for any $x\in [1,a]$ we let $$\Omega_x=\{\chi^\lambda\in\mathrm{Irr}(S_n\ |\ \gamma)\ |\ N_{p\6k}(\lambda)=x\}.$$ 
Recalling the notation introduced in \ref{not: Deltas}, we see that $\Omega_x=\{\chi^\lambda\in\mathrm{Irr}(S_n\ |\ \gamma)\ |\ \lambda\in\Delta_x\}$, and therefore that $|\Omega_x|=|\Delta_x|$, for all $x\in [1,a]$.

Since $a<p^k$ we have that $\Omega_x\neq\emptyset$ for all $x\in [1,a]$, and clearly we have that 
$$\mathrm{Irr}(S_n\ |\ \gamma)=\Omega_1\cup\Omega_2\cup\cdots\cup\Omega_a.$$

We first claim that for every $x\in [1,a-1]$ and every $\chi^\lambda\in \Omega_x$ we have that $\chi^\lambda(1)\geq \eta(1)$, for every $\eta\in\mathrm{Irr}_{p'}(N_n\ |\ \gamma)$. 
This is particularly easy to prove when the prime $p=3$. In fact, in this case we have that $a=2$ and that $\mathrm{Irr}(S_n\ |\ \gamma)=\Omega_1\cup\Omega_2$. Hence we have to consider only the case $x=1$. 
As observed at the end of Notation \ref{not: Deltas} we have that $\Delta_1\subseteq B_{n}(r+3^k)$. Since  $r+3^k\leq n-2$ we deduce that $\Delta_1\subseteq B_{n}(n-2)$. Using Lemma \ref{lem: lemma333} we see that the claim holds in this case. 
Let us now assume that $p\geq 5$. As observed at the end of Notation \ref{not: Deltas} we have that for every $x\in [1,a-1]$ if $\chi^\lambda\in\Omega_x$ then $\lambda\in\mathcal{B}_n(n-p^k)$. In particular, using Theorem \ref{thm: GL3} together with the fact that $p^k>a(p^{k-1}+p^{k-2})$, we deduce that $\lambda\in \mathcal{B}_n(n-a(p^{k-1}+p^{k-2}))\subseteq\Omega(\mathcal{X}^\star_{(n)})$. 
It follows that $[(\chi^\lambda)_{P_n}, \mathcal{X}^\star_{(n)}]\neq 0$ and therefore that $\chi^\lambda(1)\geq (p-1)^{ak}a!(p-1!)$, by Lemma \ref{lem: 11}. 
Using Corollary \ref{cor: maxdegN} we conclude that for every $x\in [1,a-1]$ and every $\chi^\lambda\in \Omega_x$ we have that $\chi^\lambda(1)\geq \eta(1)$, for every $\eta\in\mathrm{Irr}_{p'}(N_n)$.
This establishes our claim. 
To conclude, we also observe that $|\Omega_a|=|\Delta_a|=2p^k$, by Lemma \ref{lem: sizeDelta}. 

Let us now turn to analyze the local part. 
Let $Y_0(k)$ and $R(k)$ be the subsets of $\mathrm{Irr}_{p'}(N_{ap^k})$ defined in Notation \ref{not: subsets} and let $L:=\mathrm{Lin}(N_{ap^k})$. We now define $\overline{L}=\{\theta\times \chi^\gamma\ |\ \theta\in L\}$. Similarly we let $\overline{Y_0(k)}=\{\theta\times \chi^\gamma\ |\ \theta\in Y_0(k)\}$ and $\overline{R(k)}=\{\theta\times \chi^\gamma\ |\ \theta\in R(k)\}$. Of course $|\overline{L}|=|\mathrm{Lin}(N_{ap^k})|$, $|\overline{Y_0(k)}|=|Y_0(k)|$ and $|\overline{R(k)}|=|R(k)|$. We refer the reader to the end of Notation \ref{not: subsets} for the values of $|Y_0(k)|$ and $|R(k)|$, and we recall that $|L|=2(p-1)^k$, by Corollary \ref{cor: linN}.
%
%
Since $(p-1)\6{k}(2+k(p-1)^{k-1})>2p\6k$ it is very easy to verify that $$|\Omega_a|=2p^k\leq |\overline{L}\cup\overline{Y_0(k)}\cup \overline{R(k)}|.$$

Let $\lambda_0=(\gamma_1+ap^k,\gamma_2,\ldots, \gamma_t)$, $\lambda_1=(\gamma_1,\ldots, \gamma_t, 1^{ap^k})$ and let $\chi_j=\chi^{\lambda_j}$ for each $j\in [0,1]$. 
From \cite[Theorem 3.3]{OlssonBook} we see that $\chi_0, \chi_1\in\Omega_a$. Moreover using Lemma \ref{lem: bess} we have that for every $\lambda\in\mathcal{P}(n)$ such that $\chi^\lambda\in\Omega_a\smallsetminus\{\chi_0,\chi_1\}$, there exists a $ap^k$-hook $h_\lambda\in\mathcal{H}(ap^k)\setminus\{(ap^k), (1^{ap^k})\}$ such that $(\chi^\lambda)_{Y}$ has $\chi^{h_\lambda}\times\chi^\gamma$ as an irreducible constituent. Here $Y$ denotes the Young subgroup $S_{ap^k}\times S_{r}$ in $S_n$. 
From Lemma \ref{lem: Rasala1} it follows that $$\chi^{\lambda}(1)\geq \chi^{h_\lambda}(1)\cdot\chi^\gamma(1)\geq (ap^k-1)\cdot\chi^\gamma(1),\ \text{for all}\ \chi^\lambda\in\Omega_a\smallsetminus\{\chi_0,\chi_1\}.$$
In particular we have that $$\chi^\lambda(1)\geq \eta(1),\ \text{for all}\ \chi^\lambda\in\Omega_a\smallsetminus\{\chi_0,\chi_1\}\ \text{and all}\ \eta\in \overline{L}\cup \overline{Y_0(k)}\cup\overline{R(k)}.$$ Here we used items (b) and (c) of Notation \ref{not: subsets}, where we observed that $\theta(1)=a$ for every $\theta\in Y_0(k)$ and that $\zeta(1)=a(p-1)$ for every $\zeta\in R(k)$. 

The above discussion shows that there exists an injective map $\varphi: \Omega_a\rightarrow\overline{L}\cup \overline{Y_0(k)}\cup\overline{R(k)}$ such that $\varphi(\{\chi_0, \chi_1\})\subseteq\overline{L}$. Since $|\mathrm{Irr}(S_n\ |\ \gamma)|=|\mathrm{Irr}_{p'}(N_n\ |\ \gamma)|$ by Remark \ref{rem: McKay2}, we deduce that there exists a bijection $\varepsilon_\gamma: \mathrm{Irr}(S_n\ |\ \gamma)\rightarrow\mathrm{Irr}_{p'}(N_n\ |\ \gamma)$ such that $\varepsilon_\gamma(\chi)=\varphi(\chi)$ for every $\chi	\in \Omega_a$.
We conclude by showing that such a bijection satisfies the desired inequality $\varepsilon_{\gamma}(\chi)(1)\leq\chi(1)$ for every $\chi\in \mathrm{Irr}(S_n\ |\ \gamma)$.
In fact, if $\chi\in\mathrm{Irr}(S_n\ |\ \gamma)\smallsetminus\Omega_a$ then $\chi\in\Omega_x$ for some $x\in [1,a-1]$ and we have seen that $\chi(1)\geq \eta(1)$ for all $\eta\in \mathrm{Irr}_{p'}(N_n\ |\ \gamma)$. Hence we certainly have that $\varepsilon_{\gamma}(\chi)(1)\leq\chi(1)$.
Similarly, we have observed that if $\chi\in \Omega_a\smallsetminus\{\chi_0,\chi_1\}$ then $\eta(1)\leq\chi(1)$ for all $\eta\in \overline{L}\cup \overline{Y_0(k)}\cup\overline{R(k)}$.
Hence, we certainly have that $\varepsilon_{\gamma}(\chi)(1)\leq\chi(1)$, for every $\chi\in \Omega_a\smallsetminus\{\chi_0,\chi_1\}$. 
Finally, if $\chi\in\{\chi_0, \chi_1\}$ then $\chi(1)\geq \chi^\gamma(1)$, because $\gamma\subseteq\lambda_j$ for all $j\in[0,1]$. On the other hand, we know that $\varepsilon_\gamma(\chi)=\varphi(\chi)\in \overline{L}$ and therefore $\varepsilon_\gamma(\chi)(1)=\chi^\gamma(1)$. This concludes the proof. 
\end{proof}

We are now ready to prove Theorem \ref{thm: MAIN} of the introduction. 

\begin{thm}\label{thm: B2}
Let $n\in\mathbb{N}$ and let $p$ be a prime number. Then there exists a bijection $$\varepsilon_n: \mathrm{Irr}_{p'}(S_n)\rightarrow \mathrm{Irr}_{p'}(N_n),$$
such that $\varepsilon_n(\chi)(1)\leq\chi(1)$, for all $\chi\in \mathrm{Irr}_{p'}(S_n).$
\end{thm}
\begin{proof}
As mentioned in the introduction, if $p=2$ then the statement holds. 
For this reason we can assume that $p\geq 3$. Let $n=a_0+a_1p^{k_1}+a_2p^{k_2}+\cdots+a_tp^{k_t}$ be the $p$-adic expansion of $n$. 
Here $t\in\mathbb{N}$ is the $p$-adic length of $n$, $k_t>k_{t-1}>\cdots>k_1\geq 1$, $a_0\in [0,p-1]$ and $a_j\in [1,p-1]$ for every $j\in [1,t]$. 
We proceed by induction on $t$. 

If $t=1$ then $n=a_1p^{k_1}+a_0$, and the statement holds by Theorem \ref{thm: base}. Let us now assume that $t\geq 2$ and let us denote by $r$ the integer $n-a_tp^{k_t}$. 
The inductive hypothesis guarantees the existence of a bijection 
$\varepsilon_{r}:\mathrm{Irr}_{p'}(S_r)\rightarrow \mathrm{Irr}_{p'}(N_r),$
such that $\varepsilon_r(\psi)(1)\leq\psi(1)$, for all $\psi\in\ \mathrm{Irr}_{p'}(S_r).$ 
From now on, to ease the notation, we let $k=k_t$ and $a=a_t$. Notice that $k\geq 2$ (since $t\geq 2$) and that $r<p^{k}$. The case $a=1$ is dealt with using Proposition \ref{prop: auguale1}. Hence we can assume that $a\geq 2$.


In light of Proposition \ref{prop: strategy}, to prove the theorem it is enough to find for every $\gamma\in\mathcal{P}_{p'}(r)$, a bijection 
$\varepsilon_\gamma: \mathrm{Irr}(S_n\ |\ \gamma)\rightarrow\mathrm{Irr}_{p'}(N_n\ |\ \varepsilon_r(\chi\6\gamma))$ such that $\varepsilon_\gamma(\chi)(1)\leq\chi(1)$, for all $\chi\in\mathrm{Irr}(S_n\ |\ \gamma).$
We now ease the notation and write $\mathrm{Irr}_{p'}(N_n\ |\ \gamma)$ to denote the set $\mathrm{Irr}_{p'}(N_n\ |\ \varepsilon_r(\chi\6\gamma))$.

\smallskip

\textbf{Case 1.} Let us first suppose that $a=2$. In this case $N_{2p^k}=N_{p^k}\wr S_2$ and therefore Proposition \ref{prop: maxdegAP} shows that $\theta(1)\leq (p-1)^{2k}\cdot\varepsilon_r(\chi^\gamma)(1)$ for every $\theta\in\mathrm{Irr}_{p'}(N_n\ |\ \gamma)$. Given $j\in [0,2p^k-1]$ we let $h_j=(2p^k-j,1^j)$. 
Since $\chi^{h_j}(1)={2p^k-1\choose j}$, it is elementary to see that $\chi^{h_j}(1)\geq (p-1)^{2k}$, for every $j\in [2,2p^k-3]$. Using a similar strategy to that used in the proof of Theorem \ref{thm: base}, 
for any $x\in [1,2]$ we let $\Omega_x=\{\chi^\lambda\in\mathrm{Irr}(S_n\ |\ \gamma)\ |\ N_{p^k}(\lambda)=x\}$,
and we notice that 
$$\mathrm{Irr}(S_n\ |\ \gamma)=\Omega_1\cup\Omega_2.$$
We observe that $\Omega_2$ consists of all those characters labelled by partitions obtained by adding a unique $2p^k$-hook to $\gamma$. In particular $|\Omega_2|=2p^k$ and by Lemma \ref{lem: bess} we have that 
$$\Omega_2=\{\chi^{\lambda_0}, \chi^{\lambda_1},\ldots,\chi^{\lambda_{2p^k-1}}\},\ \text{where},\ [(\chi^{\lambda_j})_Y, \chi^{h_j}\times\chi^\gamma]\neq 0,\ \text{for all}\ j\in [0,2p^k-1].$$
Here $Y$ denotes the Young subgroup $S_{2p^k}\times S_r$ of $S_n$.
It follows that for any $j\in [2,2p^k-3]$ and for any $\zeta\in\mathrm{Irr}_{p'}(N_n\ |\ \gamma)$ we have that 
$$\chi^{\lambda_j}(1)\geq \chi^{h_j}(1)\chi^\gamma(1)\geq (p-1)^{2k}\cdot\chi^\gamma(1)\geq   (p-1)^{2k}\cdot\varepsilon_r(\chi^\gamma)(1)\geq\zeta(1).$$

On the other hand, let $X=(X_\gamma)\6{+p\6k}$
be the $\beta$-set for $\gamma$ corresponding to the $p^k$-abacus configuration having first empty position labelled by $p\6k$ (position $1$ on runner $0$). In particular we have that $$X=[0,p\6{k}-1]\cup\{y_1,y_2,\ldots, y_s\},$$
for some $p\6k<y_1<y_2<\cdots<y_s$. Since $|\gamma|\leq r<p\6k$ we deduce that $y_s\leq 2p\6k-1$.
From Proposition~\ref{prop: removebetaset}, we know that a partition \(\lambda\) labels an irreducible character in \(\Omega_1\) if and only if a \(p^k\)-abacus configuration for \(\lambda\) is obtained from that of \(X\) by sliding down two beads, each lying in a distinct runner of \(X\), by exactly one row. Since \(\gamma\) is a \(p^k\)-core, for every \(\ell \in [0, p^k - 1]\), there is a unique bead on runner \(\ell\) that can be moved down one row.
For this reason, for every pair of distinct numbers $0\leq x<y\leq p^k-1$ we write $\lambda_{x,y}$ for the partition corresponding to the $p\6k$-abacus configuration obtained from the abacus configuration of $X$ by sliding down beads on runners $x$ and $y$. 
We obtain that 
$$\Omega_1=\{\chi^{\lambda_{x,y}}\ |\ 0\leq x<y\leq p^k-1\}.$$

It is useful to notice that given any $y\in [1,p\6{k}-1]$, then $\lambda_{0,y}$ is a partition of $2p\6{k}+r$ whose Young diagram $Y(\lambda_{0,y})$ is obtained by adding a $p\6{k}$-hook to $Y(\gamma\cup (1\6{p\6k}))$.
Similarly, given any $x\in [0,p\6{k}-1]\smallsetminus\{y_s\}$, then $\lambda_{x,y_s}$ is a partition of $2p\6{k}+r$ whose Young diagram $Y(\lambda_{x,y_s})$ is obtained by adding a $p\6{k}$-hook to $Y(\gamma+(p\6k))$.

Let $\lambda_{x,y}$ be such that $x\neq 0$ and $y\neq y_s$. Then $$(\lambda_{x,y})_1\leq \gamma_1-1+p^k\leq p^k-2+p^k=2p^k-2$$ and similarly $$\ell(\lambda_{x,y})\leq \ell(\gamma)-1+p^k\leq p^k-2+p^k=2p^k-2.$$
In other words $\lambda_{x,y}\in\mathcal{B}_{n}(2p^k-2)$. Since $\gamma\subseteq\lambda_{x,y}$ there exists $\mu\in\mathcal{P}(2p^k)$ such that $[(\chi^{\lambda_{x,y}})_{S_{2p\6{k}}\times S_r}, \chi^{\mu}\times\chi^{\gamma}]\neq 0$. 
From Lemma \ref{lem: LR} we deduce that $\mu\subseteq\lambda$ and therefore that $\mu\in\mathcal{B}_{2p^k}(2p^k-2)$. Using Lemma \ref{lem: Rasala2} we deduce that $\chi^\mu(1)\geq p^k(2p^k-3)$. Since $p\geq 3$ it is obvious that $p^k(2p^k-3)>(p-1)\6{2k}$. Thus, we conclude that whenever $x\neq 0$ and $y\neq y_s$ then 
$\chi^{\lambda_{x,y}}(1)\geq \zeta(1)$, for every $\zeta\in\mathrm{Irr}_{p'}(N_n\ |\ \gamma)$. Now, we denote by $\Delta$ the subset of $\mathrm{Irr}_{p'}(N_n\ |\ \gamma)$ consisting of those irreducible character whose degree has not been analyzed yet. More precisely we let $$\Delta=\{\chi\6{\lambda_i}, \chi\6{\lambda_{0,y}},\chi\6{\lambda_{x,y_s}}\ |\ i\in\{0,1,2p\6k-2,2p\6k-1\},\ y\in [1,p\6k-1],\ x\in [0,p\6k-1]\smallsetminus\{y_s\}\}.$$
It is clear that $|\Delta|=2p\6k+2$.

Recall that $N_{2p\6k}=N_{p\6k}\wr S_2$ and hence that $|\mathrm{Lin}(N_{2p^k})|=2(p-1)^k$. With this in mind, we let $M$ and $R$ be the subsets of $\mathrm{Irr}_{p'}(N_{2p^k})$ defined by 
$$M=\{\big(\phi_1\times\phi_2\big)^{N_{2p^k}}\ |\ \phi_1,\phi_2\in\mathrm{Lin}(N_{p^k}),\ \text{and}\ \phi_1\neq\phi_2\},$$
$$R=\{\big(\phi\times\eta\big)^{N_{2p^k}}\ |\ \phi\in\mathrm{Lin}(N_{p^k}),\ \text{and}\ \eta\in\mathrm{QLin}(N_{p^k})\}.$$
Recalling Notation \ref{not: subsets}, we observe that $M$ and $R$ coincide with $Y_0(k)$ and $R(k)$, respectively, in the specific case where $a=2$.
In fact we have that every character in $M$ has degree equal to $2$ and that every character in $R$ has degree equal to $2(p-1)$. Moreover, we have that
$$|M|= {(p-1)^k\choose 2}\ \ \text{and that}\ \ |R|=k\cdot (p-1)^{2k-1}.$$
We now define $\overline{L}=\{\theta\times \varepsilon_r(\chi^\gamma)\ |\ \theta\in\mathrm{Lin}(N_{2p^k})\}$, $\overline{M}=\{\theta\times \varepsilon_r(\chi^\gamma)\ |\ \theta\in M\}$ and similarly $\overline{R}=\{\theta\times \varepsilon_r(\chi^\gamma)\ |\ \theta\in R\}$. 
Of course $\overline{L}\cup\overline{M}\cup\overline{R}\subseteq \mathrm{Irr}_{p'}(N_n\ |\ \gamma)$ and
$|\overline{L}|=|\mathrm{Lin}(N_{2p^k})|$, $|\overline{M}|=|M|$ and $|\overline{R}|=|R|$.
In particular it is easy to verify that $$|\overline{L}\cup\overline{M}\cup\overline{R}|=2(p-1)^k+ {(p-1)^k\choose 2}+k\cdot (p-1)^{2k-1}\geq 2p\6k+2=|\Delta|.$$
This implies that there certainly exists an injective map $\varphi:\Delta\rightarrow\overline{L}\cup\overline{M}\cup\overline{R}$ such that $\varphi(\{\chi\6{\lambda_0},\chi\6{\lambda_{p\6k-1}}\})\subseteq\overline{L}$. Since $|\mathrm{Irr}(S_n\ |\ \gamma)|=|\mathrm{Irr}_{p'}(N_n\ |\ \gamma)|$ by Remark \ref{rem: McKay2},  we can extend $\varphi$ to a bijection $\varepsilon_\gamma:\mathrm{Irr}(S_n\ |\ \gamma)\rightarrow\mathrm{Irr}_{p'}(N_{n}\ |\ \gamma)$. 
It is now routine to check that $\varepsilon_\gamma(\chi)(1)\leq\chi(1)$ for every $\chi\in\mathrm{Irr}(S_n\ |\ \gamma)$.

\smallskip

\textbf{Case 2.} Finally let us consider the case $3\leq a\leq p-1$. In particular, $p\geq 5$ here. 
We denote by $Y$ the Young subgroup $S_{ap^k}\times S_r$ of $S_n$. As done for the previous case, 
for any $x\in [1,a]$ we let $\Omega_x=\{\chi^\lambda\in\mathrm{Irr}(S_n\ |\ \gamma)\ |\ N_{p^k}(\lambda)=x\}$. 
As usual, we have that 
$$\mathrm{Irr}(S_n\ |\ \gamma)=\Omega_1\cup\Omega_2\cup\cdots\cup\Omega_a.$$
Let us first fix $x\in [1,a-2]$ and let $\chi^\lambda\in\Omega_x$. Using Proposition \ref{prop: decisiva} and recalling Notation \ref{not: Deltas} we have that $\lambda\in\Delta_x\subseteq\mathcal{B}_n(r+(a-2)p^k)=\mathcal{B}_n(n-2p^k)$. 
Moreover, since $\gamma\subseteq\lambda$, there exists $\mu\in\mathcal{P}(ap^k)$ such that $[(\chi^\lambda)_Y, \chi^\mu\times\chi^\gamma]\neq 0$. Using Lemma \ref{lem: LR} we see that $\mu\subseteq\lambda$ and hence that 
$$\mu\in\mathcal{B}_{ap^k}(n-2p^k)=\mathcal{B}_{ap^k}(r+(a-2)p^k)\subseteq\mathcal{B}_{ap^k}((a-1)p^k),$$
where the last inclusion holds because $r<p^k$. 
Using Theorem \ref{thm: GL3} together with the fact that $p^k>a(p^{k-1}+p^{k-2})$, we deduce that 
$$\mu\in \mathcal{B}_{ap^k}((a-1)p^k)\subseteq \mathcal{B}_{ap^k}(a(p^k-p^{k-1}-p^{k-2}))=\mathcal{B}_{ap^k}(a\cdot m^\star(k))\subseteq\Omega(\mathcal{X}^\star_{(ap^k)}),$$ 
where $\mathcal{X}^\star_{(ap^k)}=(\mathcal{X}^\star_{k})^{\times a}\in\mathrm{Lin}(P_{ap^k})$, as explained in Definition \ref{def: starlinear}.
Hence we have that $[(\chi^\mu)_{P_{ap^k}}, \mathcal{X}^\star_{(ap^k)}]\neq 0$ and, since $k\geq 2$, we can use Lemma \ref{lem: 11} to conclude that $$\chi^{\mu}(1)\geq (p-1)^{ak}\cdot a!\cdot(p-1!).$$ 
This shows that $$\chi^\lambda(1)\geq \chi^\mu(1)\cdot\chi^\gamma(1)\geq (p-1)^{ak}\cdot a!\cdot(p-1!)\cdot\chi^\gamma(1)\geq \zeta(1),$$
for any $\zeta\in\mathrm{Irr}_{p'}(N_n\ |\ \gamma)$ (the last inequality holds by Corollary \ref{cor: maxdegN}). 

To complete the proof of the theorem it is now enough to show the existence of an injective map $\varphi: \Omega_a\cup\Omega_{a-1}\rightarrow\mathrm{Irr}_{p'}(N_n\ |\ \gamma)$, such that $\varphi(\chi)(1)\leq\chi(1)$ for every $\chi\in\Omega_a\cup\Omega_{a-1}$.
In order to do this, we are now going to consider the subsets $X_0(k), X(k), Y(k)$ and $Z(k)$ of $\mathrm{Irr}_{p'}(N_{ap^k})$ defined in Notation \ref{not: subsets}. 
We recall that from point (e) of Notation \ref{not: subsets} we know that $\theta(1)\leq a(a-1)\leq ap^k-1$, for every 
$\theta\in X(k)\cup Y(k)\cup Z(k)$. We also refer the reader to the end of Notation \ref{not: subsets} for explicit lower bounds on the sizes of $X_0(k), X(k), Y(k)$ and $Z(k)$. Finally using that $|\Omega_a\cup\Omega_{a-1}|=|\Delta_a\cup\Delta_{a-1}|$ together with Lemma \ref{lem: sizeDelta}, we point out that 
$$|\Omega_a\cup\Omega_{a-1}|=\begin{cases}
		2p\6{2k}+p\6k & \text{if } a=3,\\
		
		2p\6{2k}+2p\6k & \text{if }  a\geq 4.
\end{cases}
$$

\smallskip

\noindent{\textbf{- Subcase (2.1)}} When $a\geq 4$, Lemma \ref{lem: Pediconi} below guarantees that $|X(k)\cup Y(k)\cup Z(k)|\geq |\Omega_{a}\cup\Omega_{a-1}|$.
Similarly, for $a=3$ and $p\geq 7$ or for $a=3$, $p\geq 5$ and $k\geq 3$, then Lemma \ref{lem: Pediconi3} gives that $|X(k)\cup Y(k)\cup Z(k)|\geq |\Omega_{a}\cup\Omega_{a-1}|$.
In these cases we have that $$|\overline{X(k)}\cup \overline{Y(k)}\cup \overline{Z(k)}|\geq |\Omega_{a}\cup\Omega_{a-1}|,$$ where for any $L\in\{X(k),Y(k),Z(k)\}$ we let $\overline{L}$ be the subset of $\mathrm{Irr}_{p'}(N_n\ |\ \gamma)$ defined by $\overline{L}=\{\theta\times\varepsilon_{r}(\chi^\gamma)\ |\ \theta\in L\}.$
To conclude, we let $\lambda_0=\gamma+(ap^k)$ and $\lambda_1=\gamma\cup (1^{ap^k})$. Clearly $\chi^{\lambda_0},\chi^{\lambda_1}\in\Omega_a\cup\Omega_{a-1}$, and from Lemma \ref{lem: LR} we have that for every 
$\chi^\lambda\in\Omega_a\cup\Omega_{a-1}\smallsetminus\{\chi^{\lambda_0},\chi^{\lambda_1}\}$ there exists $\mu\in \mathcal{P}(ap^k)\smallsetminus\{(ap^k), (1^{ap^k})\}$ such that $[(\chi^\lambda)_{S_{ap\6k}\times S_r}, \chi^\mu\times\chi^\gamma]\neq 0$. 
Since $\chi^\mu(1)\geq ap^k-1$ by Lemma \ref{lem: Rasala1}, we have that $\chi^\lambda(1)\geq (ap^k-1)\cdot\chi^\gamma(1)$ for every $\chi^\lambda\in\Omega_a\cup\Omega_{a-1}\smallsetminus\{\chi^{\lambda_0},\chi^{\lambda_1}\}$.

We deduce that there exists an injective map $\varphi: \Omega_a\cup\Omega_{a-1}\rightarrow \overline{X(k)}\cup \overline{Y(k)}\cup \overline{Z(k)}$ such that $\varphi(\{\chi^{\lambda_0},\chi^{\lambda_1}\})\subseteq\overline{X_0(k)}$. 
Since $|\mathrm{Irr}(S_n\ |\ \gamma)|=|\mathrm{Irr}_{p'}(N_{n}\ |\ \gamma)|$ by Remark \ref{rem: McKay2}, we can extend $\varphi$ to a bijection $\varepsilon_\gamma:\mathrm{Irr}(S_n\ |\ \gamma)\rightarrow\mathrm{Irr}_{p'}(N_{n}\ |\ \gamma)$. 
We observe that if $\chi\in\mathrm{Irr}(S_n\ |\ \gamma)\smallsetminus (\Omega_a\cup\Omega_{a-1})$ then $\chi\in\Omega_x$ for some $x\in [1,a-2]$ and therefore $\chi(1)\geq \zeta(1),$
for any $\zeta\in\mathrm{Irr}_{p'}(N_n\ |\ \gamma)$. In particular $\chi(1)\geq \varepsilon_\gamma(\chi)(1)$. 
On the other hand, if $\chi\in\Omega_a\cup\Omega_{a-1}\smallsetminus\{\chi^{\lambda_0},\chi^{\lambda_1}\}$ then remark $(e)$ of Notation \ref{not: subsets} used together with the inductive hypothesis imply that $$\chi(1)\geq (ap^k-1)\cdot\chi^\gamma(1)\geq\theta(1)\varepsilon_r(\chi^\gamma)(1)=(\theta\times\varepsilon_r(\chi^\gamma))(1),$$ for every $\theta\in X\cup Y\cup Z$. 
Thus we have that $\chi(1)\geq \zeta(1)$, for any $\zeta\in \overline{X(k)}\cup \overline{Y(k)}\cup \overline{Z(k)}$. In particular $\chi(1)\geq\varphi(\chi)(1)=\varepsilon_\gamma(\chi)(1)$. 
Finally, if $\chi\in\{\chi^{\lambda_0},\chi^{\lambda_1}\}$ then $\varepsilon_\gamma(\chi)=\varphi(\chi)\in \overline{X_0(k)}$ and therefore we have that $\varepsilon_\gamma(\chi)(1)=\varepsilon_r(\chi\6\gamma)(1)\leq\chi^\gamma(1)\leq \chi(1)$, as desired. 

\smallskip

\noindent{\textbf{- Subcase (2.2)}} Unfortunately when $a=3$, $p=5$ and $k=2$ we have that $$|\overline{X(2)}\cup \overline{Y(2)}\cup \overline{Z(2)}| < |\Omega_{3}\cup\Omega_{2}|.$$ For this reason, the argument used in Subcase (2.1) above can not be repeated.  
This problem occurs only for symmetric groups $S_n$ of rank $75\leq n\leq 99$ and it
is easily fixed by replacing $Z(2)$ with the set $Z(2)\cup A$, where $A$ is defined as follows: 
$$A=\{(\theta\times\eta_1\times\eta_2)^{N_{3p^k}}\ |\ \theta\in\mathrm{Irr}(N_{p^k}),\ \text{with}\ \theta(1)=4,\eta_1,\eta_2\in\mathrm{Lin}(N_{p^k}),\ \text{with}\ \eta_1\neq\eta_2\}.$$
The proof is then performed exactly as in Subcase (2.1).
\end{proof}

\begin{rem}\label{rem: finesec4}
%
In Cases 1 and 2 of the proof of Theorem \ref{thm: B2}, we consider integers of the form \( n = ap^k + r \), where the \( p \)-adic length of \( n \) is at least 2. Consequently, we have \( k \geq 2 \). This condition on \( k \) is essential for Case 2, as it relies on invoking Lemma \ref{lem: 11}.  

However, we emphasize that the situation in Case 1 is different. In fact, the argument used in Case 1 remains valid for any \( k \in \mathbb{N} \), including the case \( k = 1 \).

Specifically, Case 1 of the proof of Theorem \ref{thm: B2} shows that for every \( k \in \mathbb{N} \) and \( r \in [0, p^k - 1] \), if we set \( n = 2p^k + r \), then for any \( \gamma \in \mathcal{P}_{p'}(r) \), there exists a bijection  
\[
\varepsilon_\gamma: \mathrm{Irr}(S_n\mid \gamma) \rightarrow \mathrm{Irr}_{p'}(N_n\mid \varepsilon_r(\chi^\gamma))
\]
such that \( \varepsilon_\gamma(\chi)(1) \leq \chi(1) \) for all \( \chi \in \mathrm{Irr}(S_n \mid \gamma) \).

This observation will be useful in the proof of Proposition \ref{prop: base1} below. 
\end{rem}

\section{The $ap+a_0$ case}\label{sec: appendix}

As mentioned at the start of the proof of Theorem \ref{thm: base}, given $n\in\mathbb{N}$ of $p$-adic length equal to $1$, 
it is not possible to treat the case $n<p^2$ with the same techniques used to deal with the case $n\geq p^2$. The main reason for this being that, when $n<p^2$, we are not able to invoke Lemma \ref{lem: 11}. We devote this section to treat this remaining case.

When $p\in\{3, 5, 7\}$ then $n\leq 48$ and direct computations show that Theorem \ref{thm: base} holds in these cases. 
From now on we let $p$ be a fixed prime greater than or equal to $11$, we let $a\in [1,p-1]$ and we let $r\in [0,p-1]$.
The symbol $n$ will always denote the integer $$n=ap+r.$$
We start with some technical lemmas that will be used later in the proof of Proposition \ref{prop: base1}. 

\begin{lem}\label{lem: A1}
Let $\mu\in\mathcal{B}_{ap}(ap-3a)$, then $\chi^{\mu}(1)\geq (p-1)^a\cdot a!$. In particular, $\chi^{\mu}(1)\geq\theta(1)$, for every $\theta\in\mathrm{Irr}_{p'}(N_{ap})$. 
\end{lem}
\begin{proof}
Since $p\geq 11$, using \cite[Result 1 and Theorem F]{Rasala} we observe that every $\eta\in\mathcal{B}_p(p-3)$ is such that $$\chi^\eta(1)\geq \frac{p(p-1)(p-5)}{6}\geq p(p-1).$$ 
Moreover, \cite[Proposition 3.3]{GL2} shows that 
there exists $\mu_1,\ldots,\mu_a\in\mathcal{B}_p(p-3)$ such that $\mathrm{LR}(\mu; \mu_1,\ldots, \mu_a)\neq 0$. We conclude that $$\chi^{\mu}(1)\geq \chi^{\mu_1}(1)\cdots\chi^{\mu_a}(1)\geq p^a\cdot(p-1)^a\geq (p-1)^a\cdot a!.$$ The second statement is immediate from Proposition \ref{prop: maxdegAP}.
\end{proof}

Before stating the next lemma, we recall that for any $\gamma\in\mathcal{P}(r)$ we denote by $\mathrm{Irr}(S_n\ |\ \gamma)$ the subset of $\mathrm{Irr}(S_n)$ consisting of irreducible characters labelled by partitions whose $p$-core is equal to $\gamma$. We also remind the reader that we let $\mathrm{Irr}_{p'}(N_n\ |\ \gamma)=\{\theta\times\chi^\gamma\ |\ \theta\in\mathrm{Irr}_{p'}(N_{ap})\}$. It is important to keep in mind that 
$$|\mathrm{Irr}(S_n\ |\ \gamma)|=|\mathrm{Irr}_{p'}(N_n\ |\ \gamma)|.$$
This is a direct consequence of Lemma \ref{lem: McKay}.


\begin{lem}\label{lem: A2}
Let $\gamma\in\mathcal{P}(r)$ and let $\lambda\in\mathcal{B}_{n}(ap-3a)$ be such that $\chi^{\lambda}\in\mathrm{Irr}(S_n\ |\ \gamma)$. 
Then  $\chi^{\lambda}(1)\geq (p-1)^a\cdot a!\cdot \chi^\gamma(1)$. In particular, $\chi^{\lambda}(1)\geq\theta(1)$, for every $\theta\in\mathrm{Irr}_{p'}(N_{ap}\ |\ \gamma)$. 
\end{lem}
\begin{proof}
Since $C_p(\lambda)=\gamma$ we have that $\gamma\subseteq\lambda$ and therefore there exists $\mu\in\mathcal{P}(ap)$ such that $\mathrm{LR}(\lambda; \mu,\gamma)\neq 0$. From Lemma \ref{lem: LR} we deduce that $\mu\subseteq\lambda$ and hence that $\mu\in\mathcal{B}_{ap}(ap-3a)$. Lemma \ref{lem: A1} now implies that 
$$\chi^\lambda(1)\geq\chi^\mu(1)\chi^\gamma(1)\geq (p-1)^a\cdot a!\cdot \chi^\gamma(1).$$
The second statement is an immediate consequence of Proposition \ref{prop: maxdegAP}.
\end{proof}

We are now ready to fill the gap in the proof of Theorem \ref{thm: base}, by dealing with positive integers strictly smaller than $p^2$.

\begin{prop}\label{prop: base1}
Let $n=ap+r$ for some $a\in [1,p-1]$ and some $r\in [0,p-1]$. There exists a bijection 
$\varepsilon_n:\mathrm{Irr}_{p'}(S_n)\rightarrow\mathrm{Irr}_{p'}(N_n),$ such that $\varepsilon_n(\chi)(1)\leq\chi(1)$, for every $\chi\in\mathrm{Irr}_{p'}(S_n)$. 
\end{prop}
\begin{proof}
Let $\gamma\in\mathcal{P}(r)$. As explained by Proposition \ref{prop: strategy} it is enough to show the existence of a bijection $\varepsilon_\gamma$ between
$\mathrm{Irr}(S_n\ |\ \gamma)$ and $\mathrm{Irr}_{p'}(N_n\ |\ \gamma)$ such that $\varepsilon_{\gamma}(\chi)(1)\leq\chi(1)$ for every $\chi\in \mathrm{Irr}(S_n\ |\ \gamma)$.
By Proposition \ref{prop: auguale1} we can assume that $a\in [2,p]$. 
Moreover, as explained in Remark \ref{rem: finesec4}, the case $a=2$ can be dealt with exactly the same argument used in Case $1$ of the proof of Theorem \ref{thm: B2}. Hence from now on we work under the assumption that $a\in [3,p]$. 

As usual we first show that many irreducible characters in $\mathrm{Irr}(S_n\ |\ \gamma)$ have degree larger than the largest degree of any character in $\mathrm{Irr}_{p'}(N_n\ |\ \gamma)$. More precisely, observe that if $\chi\6\lambda\in\mathrm{Irr}(S_n\ |\ \gamma)$ for some partition $\lambda\in\mathcal{B}_n(ap-3a)$, then $\chi\6\lambda(1)\geq \eta(1)$ for every $\eta\in\mathrm{Irr}_{p'}(N_n\ |\ \gamma)$, by Lemma \ref{lem: A2}.
Let us denote by $\Delta$ the subset of $\mathcal{P}(n)$ defined by $$\Delta=\{\lambda\in\mathcal{P}(n)\ |\ \lambda\notin\mathcal{B}_n(ap-3a),\ \text{and}\ \chi\6\lambda\in\mathrm{Irr}(S_n\ |\ \gamma)\}.$$
Since $|\mathrm{Irr}(S_n\ |\ \gamma)|=|\mathrm{Irr}_{p'}(N_n\ |\ \gamma)|$,
to conclude the proof it is enough to find an injective map $\varphi:\Delta\rightarrow\mathrm{Irr}_{p'}(N_n\ |\ \gamma)$, such that $\varphi(\lambda)(1)\leq \chi^\lambda(1)$ for every $\lambda\in\Delta$. 

Recalling Notation \ref{not: Deltas}, for any $x\in [1,a]$ we let $\Delta_x$ be the subset of $\mathcal{P}(n)$ defined by 
$$\Delta_x=\{\lambda\in\mathcal{P}(n)\ |\ C_p(\lambda)=\gamma,\ \text{and}\ N_p(\lambda)=x\}.$$
The main strategy we will adopt is to show that there exists some \( y \in [1, a] \) such that \( \Delta \subseteq \Delta_y \cup \Delta_{y+1} \cup \cdots \cup \Delta_a \), and that there exists an injective map
\[
\varphi: \Delta_y \cup \Delta_{y+1} \cup \cdots \cup \Delta_a \rightarrow \mathrm{Irr}_{p'}(N_n \mid \gamma)
\]
satisfying \( \varphi(\lambda)(1) \leq \chi^\lambda(1) \) for every \( \lambda \in \Delta_y \cup \Delta_{y+1} \cup \cdots \cup \Delta_a \). The exact value of \( y \) will vary, depending on the relative sizes of \( 3a \) and \( p \). For this reason we split the proof into three main cases. 

\smallskip

\textbf{Case 1.} Suppose that $9\leq 3a\leq p$. In this case we have that $r+3a\leq 2p$ and therefore we have that $\mathcal{B}_n(r+(a-2)p)=\mathcal{B}_n(n-2p)\subseteq\mathcal{B}_n(ap-3a)$. This shows that if $\lambda\in\Delta$ then $N_p(\lambda)\geq a-1$, by Proposition \ref{prop: decisiva}. In particular $N_p(\lambda)\in\{a-1,a\}$ and hence we deduce that $\Delta\subseteq \Delta_{a-1}\cup\Delta_a$. 
Using Lemma \ref{lem: sizeDelta} we obtain that 
$$|\Delta_{a-1}\cup\Delta_a|=\begin{cases}
		2p\6{2}+p & \text{if } a=3,\\
		
		2p\62+2p & \text{if }  a\geq 4.
\end{cases}
$$

Let $\lambda_0=\gamma+(ap)$ and $\lambda_1=\gamma\cup (1^{ap})$. It is easy to see that $\lambda_0,\lambda_1\in\Delta$. Moreover, using Lemmas \ref{lem: LR} and \ref{lem: Rasala1}, we have that $\chi^\lambda(1)\geq (ap-1)\chi^\gamma(1)$ for every $\lambda\in \Delta\smallsetminus\{\lambda_0, \lambda_1\}$.

Let us now consider the subsets $X_0(1), X(1)$ and $Y(1)$  of $\mathrm{Irr}_{p'}(N_{ap})$ defined in Notation \ref{not: subsets}. Moreover, For any $Z\subseteq \mathrm{Irr}_{p'}(N_{ap})$ we denote by $\overline{Z}$ the subset of $\mathrm{Irr}_{p'}(N_{n}\ |\ \gamma)$ defined by 
$$\overline{Z}=\{\theta\times\chi^\gamma\ |\ \theta\in Z\}.$$
Clearly $|Z|=|\overline{Z}|$, for every $Z\subseteq\mathrm{Irr}_{p'}(N_{ap})$. 

From item (e) of Notation \ref{not: subsets} we know that $\theta(1)\leq ap-1$, for every $\theta\in X(1)\cup Y(1)$. Moreover, from the last part of Notation \ref{not: subsets} we have explicit lower bounds (and sometimes exact computations) for the sizes of the three sets $X_0(1), X(1)$ and $Y(1)$.

%
%
%
%
%
%
%
%
%

In particular, when $a\geq 4$, since $p\geq 11$ we have that $|\overline{X(1)}\cup\overline{Y(1)}|\geq |\Delta|$. Hence there exists an injective map $\varphi: \Delta\rightarrow\overline{X(1)}\cup\overline{Y(1)}$ such that $\varphi(\{\lambda_0,\lambda_1\})\subseteq\overline{X_0(1)}$. The observations collected above show that $\varphi(\lambda)(1)\leq \chi^\lambda(1)$ for every $\lambda\in\Delta$, as desired. 
Unfortunately, when $a=3$ we have that $|\overline{X(1)}\cup\overline{Y(1)}|< |\Delta|$. This problem is easily solved by considering the additional subset $V$ of $\mathrm{Irr}_{p'}(N_{3p})$ defined by
$$V=\{(\phi_1\times\phi_2\times\phi_3)^{N_{3p}}\ |\ \phi_1,\phi_2,\phi_3\in\mathrm{Lin}(N_{p})\},$$
where $\phi_1,\phi_2,\phi_3$ are always chosen to be pairwise distinct. It is now routine to check that exists an injective map $\varphi: \Delta\rightarrow\overline{X(1)}\cup\overline{Y(1)}\cup\overline{V}$ such that $\varphi(\lambda)(1)\leq \chi^\lambda(1)$ for every $\lambda\in\Delta$.

\smallskip

\textbf{Case 2.} Suppose that $p+1\leq 3a\leq 2p$. Since $p\geq 11$ we observe that $a\geq 4$. 
In this case we have that $\mathcal{B}_n(r+(a-3)p)=\mathcal{B}_n(n-3p)\subseteq\mathcal{B}_n(ap-3a)$. This shows that if $\lambda\in\Delta$ then $N_p(\lambda)\geq a-2$. In particular $N_p(\lambda)\in\{a-2, a-1,a\}$ and $\Delta\subseteq \Delta_{a-2}\cup\Delta_{a-1}\cup\Delta_a$. 
Since $a\geq 4$, from Lemma \ref{lem: sizeDelta} we know that $|\Delta_{a-1}\cup\Delta_a|=2p^2+2p$. 

\smallskip

\noindent\textbf{- Subcase (2.1)} We first assume that $a\geq 6$. 
From Lemma \ref{lem: sizeDelta} we have that 
$$|\Delta_{a-2}|=p^3+3p^2.$$
Moreover, for any $\lambda\in\Delta_{a-2}$ we claim that $$\chi^\lambda(1)\geq \frac{1}{2}ap(ap-3)\cdot\chi^\gamma(1).$$ To see this, notice first that
$\Delta_{a-2}\subseteq\mathcal{B}_n(r+(a-2)p)=\mathcal{B}_n(n-2p)$ by Proposition \ref{prop: decisiva}. 
Since $\mathcal{B}_n(n-2p)=\mathcal{B}_n(ap-(2p-r))$ we deduce that there exists $\mu\in\mathcal{B}_{ap}(ap-(2p-r))\subseteq \mathcal{B}_{ap}(ap-p)$ such that $\mathrm{LR}(\lambda; \gamma, \mu)\neq 0$. From Lemma \ref{lem: Rasala2} we have that $\chi^\mu(1)\geq \frac{1}{2}ap(ap-3)$ and therefore that $\chi^\lambda(1)\geq \frac{1}{2}ap(ap-3)\cdot\chi^\gamma(1)$, as claimed. 

We now consider the following subsets of $\mathrm{Irr}_{p'}(N_{ap})$. We let $X(1)$ and $Y(1)$ be defined exactly as in Case 1 above, and we further consider: 
$$W=\{(\mathcal{X}(\phi; \chi^\nu)\times\eta_1\times\eta_2)^{N_{ap}}\ |\ \phi,\eta_1,\eta_2\in\mathrm{Lin}(N_{p}),\ \nu\in\{(a-2), (a-3,1)\}^{\circ}\},$$
$$V=\{(\mathcal{X}(\phi; \chi^\nu)\times\mathcal{X}(\eta; \chi^\rho))^{N_{ap}}\ |\ \phi,\eta\in\mathrm{Lin}(N_{p}),\ \nu\in\{(a-2), (a-3,1)\}^{\circ},\  \rho\in\{(2)\}^{\circ}\}.$$
As usual, or any $Z\subseteq \mathrm{Irr}_{p'}(N_{ap})$ we denote by $\overline{Z}$ the subset of $\mathrm{Irr}_{p'}(N_{n}\ |\ \gamma)$ defined by 
$$\overline{Z}=\{\theta\times\chi^\gamma\ |\ \theta\in Z\}.$$
Clearly $|Z|=|\overline{Z}|$, for every $Z\subseteq\mathrm{Irr}_{p'}(N_{ap})$. 

\smallskip

We collect a few observations concerning $W$ and $V$. 

\smallskip

\noindent\textbf{(a)} In the definition of the set $W$ the linear characters $\phi$, $\eta_1$ and $\eta_2$ are chosen to be pairwise distinct, so that $\mathcal{X}(\phi; \chi^\nu)\times\eta_1\times\eta_2$ is an irreducible character of $N_{(a-2)p}\times N_p\times N_{p}\leq N_{ap}$ that induces irreducibly to $N_{ap}$. 
In particular we have that $\theta(1)\leq a(a-1)(a-3)$, for every $\theta\in W$.

\smallskip

\noindent\textbf{(b)} In the definition of the set $V$ the linear characters $\phi$ and $\eta$ are chosen to be distinct, so that $\mathcal{X}(\phi; \chi^\nu)\times\mathcal{X}(\eta; \chi^\rho)$ is an irreducible character of $N_{(a-2)p}\times N_{2p}\leq N_{ap}$ that induces irreducibly to $N_{ap}$. 
In particular we have that $\theta(1)\leq \frac{1}{2}a(a-1)(a-3)$, for every $\theta\in V$.

\smallskip

\noindent\textbf{(c)} From $(a), (b)$ we conclude that $\theta(1)\leq a(a-1)(a-3)\leq \frac{1}{2}ap(ap-3)$, for every $\theta\in W\cup V$. In particular, for every $\zeta\in\overline{W}\cup\overline{V}$ we have that 
$$\zeta(1)\leq \frac{1}{2}ap(ap-3)\cdot\chi^\gamma(1)\leq \chi^\lambda(1),\ \ \text{for all}\ \lambda\in\Delta_{a-2}.$$

\smallskip

Using Corollary \ref{cor: linN} and recalling that $a\geq 6$ it is easy to see that
$$|W|=4(p-1){p-2\choose 2}=2p^3-12p^2+22p-12,\ |V|=4\cdot2\cdot(p-1)(p-2)=8p^2-24p+16.$$
Since $p\geq 11$, a straightforward calculation shows that $|\overline{W}\cup\overline{V}|\geq |\Delta_{a-2}|$.

Arguing exactly as in Case 1, all the observations above allow us to deduce that
there exists an injective map $$\varphi: \Delta_a\cup\Delta_{a-1}\cup\Delta_{a-2}\rightarrow \overline{X(1)}\cup\overline{Y(1)}\cup\overline{W}\cup\overline{V}\subseteq\mathrm{Irr}_{p'}(N_n\ |\ \gamma),$$
such that $\varphi(\Delta_a\cup\Delta_{a-1})\subseteq  \overline{X(1)}\cup\overline{Y(1)}$, $\varphi(\Delta_{a-2})\subseteq  \overline{W}\cup\overline{V}$, and such that $\varphi(\lambda)(1)\leq \chi^\lambda(1)$ for every $\lambda\in\Delta_a\cup\Delta_{a-1}\cup\Delta_{a-2}$.

\smallskip

\noindent\textbf{- Subcase (2.2)} We now consider the case \( 4 \leq a \leq 5 \). Here, \( p \in \{11,13\} \), and the argument closely parallels that of Subcase~(2.1), with only minor adjustments required. It follows that there exists an injective map \( \varphi: \Delta \rightarrow \mathrm{Irr}_{p'}(N_n\ |\ \gamma) \) such that \( \varphi(\lambda)(1) \leq \chi^\lambda(1) \) for every \( \lambda \in \Delta \). We therefore omit the technical details.

\smallskip

\textbf{Case 3.} Suppose that $2p+1\leq 3a< 3p$. Since $p\geq 11$ we observe that $a\geq 8$. 
In this case we have that $\mathcal{B}_n(r+(a-4)p)=\mathcal{B}_n(n-4p)\subseteq\mathcal{B}_n(ap-3a)$. This shows that if $\lambda\in\Delta$ then $N_p(\lambda)\geq a-3$, by Proposition \ref{prop: decisiva}. In particular $N_p(\lambda)\in\{a-3,a-2, a-1,a\}$ and $\Delta\subseteq \Delta_{a-3}\cup\Gamma$, where $\Gamma=\Delta_{a-2}\cup\Delta_{a-1}\cup\Delta_a$. Recycling the notation and the arguments used in Cases 1 and 2, we consider the subsets 
$\overline{X(1)}, \overline{Y(1)}, \overline{W}, \overline{V}\subseteq\mathrm{Irr}_{p'}(N_n\ |\ \gamma)$ and we observe that there exists an injective map $$\varphi: \Delta_a\cup\Delta_{a-1}\cup\Delta_{a-2}\rightarrow \overline{X(1)}\cup\overline{Y(1)}\cup\overline{W}\cup\overline{V}\subseteq\mathrm{Irr}_{p'}(N_n\ |\ \gamma),$$
such that $\varphi(\Delta_a\cup\Delta_{a-1})\subseteq  \overline{X(1)}\cup\overline{Y(1)}$, $\varphi(\Delta_{a-2})\subseteq  \overline{W}\cup\overline{V}$, and such that $\varphi(\lambda)(1)\leq \chi^\lambda(1)$ for every $\lambda\in\Delta_a\cup\Delta_{a-1}\cup\Delta_{a-2}$. Let $A=\overline{X(1)}\cup\overline{Y(1)}\cup\overline{W}\cup\overline{V}$, to conclude the proof we just need to find a second injective map $f:\Delta_{a-3}\rightarrow\mathrm{Irr}_{p'}(N_n\ |\ \gamma)\smallsetminus A$ such that $f(\lambda)(1)\leq \chi^\lambda(1)$ for every $\lambda\in\Delta_{a-3}$. 
To this end, we consider the subset $Z\subseteq\mathrm{Irr}_{p'}(N_{ap})$ defined by
$$Z=\{(\mathcal{X}(\phi; \chi^\nu)\times\mathcal{X}(\eta; \chi^\rho)\times\eta_1\times\eta_2)^{N_{ap}}\ |\ \phi,\eta,\eta_1,\eta_2\in\mathrm{Lin}(N_{p}),\ \nu\in\{(a-4)\}^{\circ},\  \rho\in\{(2)\}^{\circ}\}.$$
As usual, we denote by $\overline{Z}$ the subset of $\mathrm{Irr}_{p'}(N_{n}\ |\ \gamma)$ defined by 
$$\overline{Z}=\{\theta\times\chi^\gamma\ |\ \theta\in Z\}.$$
Clearly $|Z|=|\overline{Z}|$, and $\overline{Z}\subseteq\mathrm{Irr}_{p'}(N_n\ |\ \gamma)\smallsetminus A$. 

We collect a few observations on the sets $Z$ and $\Delta_{a-3}$.

\smallskip

\noindent\textbf{(a)} In the definition of the set $Z$ the linear characters $\phi$, $\eta$, $\eta_1$ and $\eta_2$ are chosen to be pairwise distinct, so that $\mathcal{X}(\phi; \chi^\nu)\times\mathcal{X}(\eta; \chi^\rho)\times\eta_1\times\eta_2$ is an irreducible character of $N_{(a-4)p}\times N_{2p}\times N_{p}\times N_p\leq N_{ap}$ that induces irreducibly to $N_{ap}$. 
In particular we have $$\theta(1)\leq \frac{1}{2}a(a-1)(a-2)(a-3),\ \text{for every}\ \theta\in Z.$$

\smallskip

\noindent\textbf{(b)} By Proposition \ref{prop: decisiva} we have that $\Delta_{a-3}\subseteq \mathcal{B}_n(r+(a-3)p)\subseteq\mathcal{B}_n(ap-2p)$. Hence, using Lemma \ref{lem: Rasala2} and arguing exactly as in Case 2 above, we deduce that $$\chi^\lambda(1)\geq\frac{1}{2}ap(ap-3)\chi^\gamma(1),\ \text{for every}\ \lambda\in \Delta_{a-3}.$$ 

\smallskip

\noindent\textbf{(c)} Recalling that $a\geq 8$, from observations (a) and (b) we deduce that $$\theta(1)\leq\chi^\lambda(1)\ \text{for every}\ \theta\in\overline{Z}\ \text{and every}\ \lambda\in\Delta_{a-3}.$$ 

\smallskip

\noindent\textbf{(d)} Arguing exactly as in the proof of Lemma \ref{lem: sizeDelta} by counting the number of possible different $p$-quotients $Q_p(\lambda)$ of total size $a$ such that $N_p(\lambda)=a-3$, we deduce that 
$$|\Delta_{a-3}|=p^2(\frac{1}{3}p^2+3p+\frac{8}{3}).$$

\smallskip

\noindent\textbf{(e)} On the other hand, it is easy to see that $$|\overline{Z}|=|Z|=2(p-1)(p-2)(p-3)(p-4)=2p^4-20p^3+70p^2-100p+48.$$

Since $p\geq 11$, using observations (d) and (e), it is routine to check that $|\overline{Z}|\geq \Delta_{a-3}$. Therefore, using observation $(c)$ as well, we deduce that there exists an injective map $$f:\Delta_{a-3}\rightarrow \overline{Z}\subseteq\mathrm{Irr}_{p'}(N_n\ |\ \gamma)\smallsetminus A,$$ such that $f(\lambda)(1)\leq \chi^\lambda(1)$ for every $\lambda\in\Delta_{a-3}$. This concludes the proof. 
\end{proof}

\begin{rem}
Let $G$ be a finite group and $N = N_G(P)$ for $P \in \mathrm{Syl}_p(G)$. We note that Conjecture A would follow for $G$ if there existed a McKay bijection $f: \mathrm{Irr}_{p'}(G) \to \mathrm{Irr}_{p'}(N)$ compatible with character restriction, i.e., such that $[f(\chi), \chi_N] \neq 0$ for all $\chi \in \mathrm{Irr}_{p'}(G)$. However, it is important to observe that this stronger statement does not hold in general.

In the case of symmetric groups, such bijections are known to exist when $p=2$ (see, for instance, \cite[Theorem 4.3]{GKNT17}). On the other hand, they do not exist for $p \geq 5$. A small counterexample is provided by $S_5$ at the prime $p=5$, where two distinct characters $\chi^{(4,1)}, \chi^{(2,1^3)} \in \mathrm{Irr}_{5'}(S_5)$ restrict to the same irreducible character of $N$. This argument readily generalizes to $S_p$ for any prime $p \geq 5$, showing that no such bijection can exist.
The case $p=3$ is more elusive. A partial positive result was established in \cite[Theorem 4.14]{GTT}.
\end{rem}

%
%
%
%

\section{Appendix}\label{sec: 6}

This section is devoted to collecting a few elementary analytic observations used to justify some of the inequalities appearing in the article. We have chosen to present these results separately in an appendix in order to make the main body of the paper easier to read.

\begin{lem}\label{lem: analisi1}
Let $a,p,k\in\mathbb{R}$ be such that $k\geq 2$, $p\geq 5$ and $a\in [1,p-1]$. Then $$ap\6{k-1}\geq ak+(a-1)+(p-2).$$
\end{lem}
\begin{proof}
If $k=2$ then consider $f(a,p)=ap-3a-p+3$. It is easy to see that $f(a,p)=a(p-3)-(p-3)\geq 0$, for all $p\geq 5$ and all $a\in [1,p-1]$. Hence the statement holds in this case. 

Let us now assume that $k\geq 3$. Since $a\leq p-1$ we have that $$ak+(a-1)+(p-2)\leq (p-1)(k+2)-2\leq p(k+2).$$ On the other hand, $ap\6{k-1}\geq p\6{k-1}$. Let $h(k,p)=p\6{k-2}-(k+2)$ and let
 $$g(k,p)=p\6{k-1}-p(k+2)=p(p\6{k-2}-(k+2))=p\cdot h(k,p),$$
be defined on $D=\{(k,p)\ |\ k\geq 3, p\geq 5\}\subseteq \mathbb{R}\6{2}$. It is easy to see that $\frac{d}{dp}h(k,p)=(k-2)p\6{k-3}\geq 0$ for all $(k,p)\in D$. Similarly $\frac{d}{dk}h(k,p)=p\6{k-2}\mathrm{log}(p)-1\geq 0$ for all $(k,p)\in D$. It follows that for any $(k,p)\in D$ we have that $g(k,p)\geq g(3,5)=0$. We conclude that 
$$ap\6{k-1}\geq p\6{k-1}\geq  p(k+2)\geq (p-1)(k+2)-2\geq ak+(a-1)+(p-2),$$
as desired.
\end{proof}

\begin{lem}\label{lem: Pediconi}
Let $f(p,k)$ be the function of two real variables defined by $$f(p,k)=(p-1)\6{3k}+3(p-1)\6{k}-2p\6{2k}-2p\6{k}.$$
Then $f(p,k)\geq 0$ for all $p,k\in\mathbb{R}$ such that $k\geq 2$ and $p\geq 5$. 
\end{lem}
\begin{proof}
Since $2p\6{2k}+2p\6{k}\leq 4p^{2k}$ we have that $$f(p,k)\geq (p-1)\6{3k}-4p^{2k}=p^{2k}\big(\frac{(p-1)^{3k}}{p^{2k}}-4\big).$$
Let us consider the function $$h(x,y)=\frac{x^{3y}}{(x+1)^{2y}}-4,\ \text{on the domain}\ D=\{(x,y)\in\mathbb{R}^2\ |\ x\geq 4, y\geq 2\}.$$
Notice that $h(x,y)=e^{yg(x)}-4$, where $g(x)=3\mathrm{log}(x)-2\mathrm{log}(x+1)$.
Since $$g'(x)=\frac{3}{x}-\frac{2}{x+1}=\frac{x+3}{x(x+1)}>0,\ \text{for all}\ x\geq 4,$$
we deduce that the function $h(x,y)$ is increasing in both $x$ and $y$. Since $h(4,2)>0$ we deduce that $h(x,y)>0$ for every $(x,y)\in D$. We conclude that 
$$f(p,k)\geq p^{2k}\cdot h(p-1,k)\geq 0,\ \text{for every}\ k\geq 2\ \text{and every}\ p\geq 5.$$
\end{proof}

\begin{lem}\label{lem: Pediconi3}
Let $f(p,k)$ be the function of two real variables defined by $$f(p,k)=\frac{1}{6}(p-1)\6{3k}+\frac{3}{2}(p-1)\6{2k}+\frac{4}{3}(p-1)^k-2p\6{2k}-p\6{k}.$$
Then $f(p,k)\geq 0$ for all $p,k\in\mathbb{R}$ such that $k\geq 2$ and $p\geq 7$. Moreover, $f(5,k)\geq 0$, for all $k\geq 3$.
\end{lem}
\begin{proof}
Observe that $f(p,k)\geq \frac{1}{6}(p-1)^{3k}-2p^{2k}=\frac{1}{6}p^{2k}\big(\frac{(p-1)^{3k}}{p^{2k}}-12\big)$. Considering the function $h(x,y)=\frac{x^{3y}}{(x+1)^{2y}}-12$ on the domain $D=\{(x,y)\in\mathbb{R}^2\ |\ x\geq 4, y\geq 2\}$, observing that $h(4,3)\geq 0$ and that $h(6,2)\geq 0$, and arguing exactly as in the proof of Lemma \ref{lem: Pediconi} we conclude the proof. 
\end{proof}


\end{document}